\documentclass[11pt]{amsart}
\usepackage{amsmath,amssymb,amsthm,mathtools,mathrsfs,tikz,comment}
\usepackage[margin=1.15in]{geometry}
\usepackage{enumitem}
\usepackage[colorlinks=true,citecolor=blue,linkcolor=blue,urlcolor=blue]{hyperref}

\newtheorem{theorem}{Theorem}[section]
\newtheorem{proposition}[theorem]{Proposition}
\newtheorem{lemma}[theorem]{Lemma}

\theoremstyle{plain}

\newtheorem{mainresult}{Theorem}
\newtheorem{maincorollary}[mainresult]{Corollary}

\newcommand{\Pj}{\mathbb P}
\newcommand{\Kah}{\mathcal K}
\newcommand{\Ric}{\operatorname{Ric}}
\newcommand{\ddbar}{\sqrt{-1}\,\partial\bar\partial}
\newcommand{\Fcal}{\mathcal F}
\newcommand{\NA}{\mathrm{NA}}
\newcommand{\cpsc}{\mathcal K_X^{\mathrm{psc}}}
\newcommand{\ctpsc}{\mathcal T_X^+}
\newcommand{\Hpot}{\mathcal H}
\newcommand{\eps}{\varepsilon}
\newcommand{\pr}{\operatorname{pr}}

\title[PSC K\"ahler cones on compact complex surfaces]
{Positive scalar curvature K\"ahler cones on compact complex surfaces}
\author{Zehao Sha}
\date{}
\begin{document}

\begin{abstract}
We determine the positive scalar curvature K\"ahler cone of every minimal compact K\"ahler surface of Kodaira dimension $-\infty$. By the Enriques--Kodaira classification, every such surface is either $\mathbb P^2$ or a geometrically ruled surface $\mathbb P(E)\to\Sigma_g$, where $E$ is a rank-two holomorphic vector bundle over a compact Riemann surface of genus $g$. On $\mathbb P^2$, every K\"ahler class admits a K\"ahler metric of positive scalar curvature, whereas on $\mathbb P(E)$, every K\"ahler class of positive total scalar curvature admits a metric of positive scalar curvature if and only if $g\le1$ or $E$ is slope-semistable. We further prove that if a smooth compact K\"ahler surface admits a birational morphism onto a ruled surface $\mathbb P(E)\to\Sigma_g$, where $g\ge2$ and $E$ is slope-unstable, then it contains a K\"ahler class of positive total scalar curvature admitting no positive scalar curvature K\"ahler metric. Finally, for every smooth compact K\"ahler surface $X$ birational to $Y=\mathbb P^2$ or $Y=\mathbb P(E)$, we show that K\"ahler classes admitting positive scalar curvature metrics on $Y$ are exactly the pushforwards of classes admitting such metrics on $X$.
\end{abstract}

\maketitle

\section{Introduction}
\label{sec:introduction}

The existence problem for Riemannian metrics of positive scalar curvature (PSC) admits several criteria and obstructions \cite{Lichnerowicz1963,Hitchin1974,SchoenYau1979Structure,GromovLawson1980,Kobayashi1987,Schoen1984}. In K\"ahler geometry, metric existence problems are naturally formulated within a fixed K\"ahler class. Thus, even when the underlying complex manifold admits some K\"ahler metric of positive scalar curvature, it remains a separate problem to decide which K\"ahler classes contain such metrics.

Let \(X\) be a compact K\"ahler manifold and let \(\mathcal K_X\) denote the K\"ahler cone.  We define the PSC K\"ahler cone by
\[
 \mathcal K_X^{\mathrm{psc}}:=  \left\{ \alpha\in\mathcal K_X:\,\exists\, \omega \in \alpha,~\text{ s.t. }S(\omega)>0\right\}.
\]
For any \(\alpha \in \cpsc\), we call \(\alpha\) a PSC K\"ahler class. A necessary condition for a K\"ahler class \(\alpha \in \cpsc\) is the positive total scalar curvature:
\[
\int_X S(\omega)\, \omega^n =  2n\pi c_1(X)\cdot\alpha^{n-1}>0.
\]
Accordingly, one can define the positive total scalar curvature K\"ahler cone
\[
\mathcal T_X^+:=\left\{ \alpha\in\mathcal K_X: c_1(X)\cdot\alpha^{n-1}>0 \right\}.
\]
It is clear that \(\mathcal K_X^{\mathrm{psc}}\subseteq\mathcal T_X^+\). Whether the reverse inclusion holds is an open question: an affirmative answer would give a purely numerical criterion for the existence of a PSC K\"ahler metric in a prescribed K\"ahler class. Moreover, it would identify the PSC K\"ahler cone, and in particular its boundary, by the intersection number \(c_1(X)\cdot\alpha^{n-1}>0\). A closely related question was asked by Brown in \cite[Question~7.6]{Brown2026} as well.

The first purpose of this paper is to show that the answer to the above question is negative in general. More precisely, we first characterize the PSC K\"ahler cone on minimal K\"ahler surfaces of Kodaira dimension \(-\infty\). By the Enriques--Kodaira classification, such a surface is either \(\Pj^2\), a Hirzebruch surface \(\mathbb F_n\) with \(n=0\) or \(n\ge2\), or a geometrically ruled surface \(\Pj(E)\to\Sigma_g\) over a compact Riemann surface of genus \(g\ge1\), where \(E\) is a rank-two holomorphic vector bundle, and \(\Pj(E)\) denotes the projectivization of \(E\); see \cite[Chapter~VI]{BHPV2004}.

On a minimal ruled surface  \(X=\Pj(E)\xrightarrow{\pi}\Sigma_g\), let \(F:=\pi^*[\mathrm{pt}]\in H^{1,1}(X,\mathbb R)\) denote the fibre class. Define the normalized tautological class (see \cite[p.~456]{Miyaoka1987}) by
\[
\zeta:= \frac12c_1\bigl(K_{X/\Sigma_g}^{-1}\bigr) = c_1\bigl(\mathcal O_{\Pj(E)}(1)\bigr) -\frac{\deg E}{2}\,F.
\]
Moreover, \(\zeta\) depends only on the ruling and is unchanged under \(E\mapsto E\otimes L\) for \(L\) a line bundle. In particular, \(F\) and \(\zeta\) generate \(H^{1,1}(X,\mathbb R)\) and every real \((1,1)\)-class can be written uniquely as \(\alpha=2\pi(a\zeta+bF)\).

For a nonzero holomorphic subbundle \(S\subseteq E\), define \(\mu(S):=\deg S/\operatorname{rank}S\) and let \(\mu_{\max}(E)\) denote the largest slope in the Harder--Narasimhan filtration of \(E\). When \(E\) is semistable, we require \(\mu_{\max}(E)=\mu(E)\). We then define 
\begin{equation}
\Delta(E):=2\mu_{\max}(E)-\deg E.
\label{eq:intro-HN-defect}
\end{equation}
Thus \(\Delta(E)\geq0\), with equality if and only if \(E\) is slope-semistable. Moreover, \(\Delta(E)\) is also unchanged under \(E\mapsto E\otimes L\). Our first result is the following.
\begin{mainresult}
\label{thm:main-minimal}
Let \(X\) be a minimal compact K\"ahler surface of Kodaira dimension \(-\infty\).
\begin{enumerate}[label=\textup{(\roman*)}]
\item If \(X=\Pj^2\) or \(X=\mathbb F_n\) with \(n=0\) or \(n\ge2\), then
\[
\Kah_X^{\mathrm{psc}} = \mathcal T_X^+ =\Kah_X.
\]
\item If \(X=\Pj(E)\rightarrow\Sigma_g\) with \(g\ge1\), then
\begin{align*}
\mathcal T_X^+ &= \left\{ 2\pi(a\zeta+bF): a>0,\quad b>\max\left\{g-1,\frac{\Delta(E)}2\right\}a \right\},
\\
\Kah_X^{\mathrm{psc}} &= \left\{ 2\pi(a\zeta+bF): a>0,\quad b>\left(g-1+\frac{\Delta(E)}2\right)a \right\}.
\end{align*}
Consequently,
\[
\Kah_X^{\mathrm{psc}}=\mathcal T_X^+ \quad\Longleftrightarrow\quad g=1\ \text{ or }\ E\ \text{is slope-semistable}.
\]
\end{enumerate}
\end{mainresult}

For decomposable bundles \(E=\mathcal O_{\Sigma_g} \oplus L\), we also give in Appendix~\ref{app:decomposable-model} a direct momentum construction, which therefore provides a concrete geometric model for the obstruction appearing in the above result.

\vspace{2mm}

The variational characterization in \cite[Theorem~A]{Sha2026Variational} yields a general stability
obstruction to equality of the two cones. We show moreover that this obstruction persists under birational
morphisms. Applying this principle to the deformation to the normal cone of the Harder--Narasimhan section gives the following result.

\begin{mainresult}
\label{thm:main-birational-obstruction}
Let \(X\) be a smooth compact K\"ahler surface. Suppose that \(X\) admits a birational morphism
\(p:X\rightarrow Y=\Pj(E)\), where \(E\) is a slope-unstable rank-two holomorphic vector bundle over a compact Riemann surface \(\Sigma_g\) of genus \(g\ge2\). Then
\[
\Kah_X^{\mathrm{psc}} \subsetneq \mathcal T_X^+.
\]
\end{mainresult}

Theorem~\ref{thm:main-birational-obstruction} gives a necessary birational condition for equality of the two cones: if equality holds, every geometrically ruled surface over a curve of genus at least two dominated
by \(X\) must be defined by a slope-semistable vector bundle. 

\vspace{2mm}

To understand the birational behavior of the PSC K\"ahler cone more precisely, we next consider the pushforward of individual PSC K\"ahler classes. The following result determines the image of the PSC K\"ahler cone under a birational morphism to its minimal model.

\begin{mainresult}
\label{thm:ruled-descent}
Let \(p:X\to Y\) be a birational morphism of smooth compact K\"ahler surfaces. Suppose that \(Y=\Pj^2\), or that \(Y=\Pj(E)\) for a rank-two holomorphic vector bundle \(E\) over a compact Riemann surface of any
genus. Then
\[
 p_*\Kah_X^{\mathrm{psc}}=\Kah_Y^{\mathrm{psc}}.
\]
\end{mainresult}

Here \(p_*\) denotes the pushforward on real \((1,1)\)-classes. The morphism may be any finite composition of point blowups, including infinitely near centers, and the target ruled surfaces include \(\mathbb F_1\). Brown's blowup theorem \cite[Theorem~6.8]{Brown2026} gives one direction of the equality in Theorem~\ref{thm:ruled-descent}. Indeed, let \(Y=\Pj^2\) or a geometrically ruled surface, and let \(\pi:X=\operatorname{Bl}_{p_1,\ldots,p_k}Y\to Y\) be a finite sequence of point blowups, allowing infinitely near centers. For every \(\alpha\in\Kah_Y^{\mathrm{psc}}\), successive applications of Brown's theorem give
\[
 \beta=\pi^*\alpha-\sum_{i=1}^k\varepsilon_i[D_i] \in\Kah_X^{\mathrm{psc}},
\]
where \(D_i\) denotes the total transform on \(X\) of the exceptional divisor introduced at the \(i\)-th blowup, and the positive parameters \(\varepsilon_i\) are chosen sufficiently small at each stage. Theorem~\ref{thm:ruled-descent} proves the reverse implication in this setting: if \(\beta\in\Kah_X^{\mathrm{psc}}\), then \(\alpha\in\Kah_Y^{\mathrm{psc}}\), without any smallness assumption on each \(\varepsilon_i\). Together with Theorem~\ref{thm:main-minimal}, we obtain an alternative proof of Theorem~\ref{thm:main-birational-obstruction}, by lifting a class in \(\mathcal{T}^+_Y \setminus\mathcal{K}^{\mathrm{psc}}_Y\) with sufficiently small blowup parameters.

\vspace{2mm}

The birational invariance of the existence of PSC K\"ahler metrics on compact K\"ahler surfaces follows from \cite[Theorem~B]{Brown2026}. Theorem~\ref{thm:ruled-descent}, combined with \cite[Theorem~6.8]{Brown2026}, gives the following slight refinement, which preserves the pushforward of a given PSC K\"ahler class to a chosen minimal model.

\begin{maincorollary}
\label{cor:psc-birational-invariance}
Let \(p:X\to Y\) be a birational morphism of smooth compact K\"ahler surfaces, and let \(r:Y\to Y_{\mathrm{min}}\) be a birational morphism to a minimal model of \(Y\). For every \(\alpha\in\Kah_X^{\mathrm{psc}}\), there exists \(\beta\in\Kah_Y^{\mathrm{psc}}\) such that
\[
 r_*\beta=(r\circ p)_*\alpha.
\]
\end{maincorollary}
Equivalently, \(\beta-p_*\alpha\) is a real linear combination of the classes of curves contracted by \(r\). When \(Y\) is minimal, it gives \(\beta=p_*\alpha\).

\subsection*{Organization of the paper}
The paper is organized as follows.
Section~\ref{sec:preliminaries} recalls the geometry of ruled surfaces, the variational formulation of PSC K\"ahler metrics together with its non-Archimedean formulation and some intersection formulas on the total space.
Section~\ref{sec:HN-obstruction} determines the PSC K\"ahler cones of minimal surfaces and proves Theorem~\ref{thm:main-minimal}.
Section~\ref{sec:birational-propagation} studies the birational behavior for the PSC K\"ahler cone and proves Theorem \ref{thm:main-birational-obstruction} and \ref{thm:ruled-descent}.
Appendix~\ref{app:decomposable-model} gives the momentum construction on decomposable ruled surfaces and identifies the resulting invariant cone with the full PSC K\"ahler cone.

\subsection*{On the use of AI} 
The main theme of this paper, \textit{Whether every K\"ahler class with positive total scalar curvature carries a PSC K\"ahler metric} was considered by the author starting in 2025. In July 2026, the author was asking ChatGPT~5.5 to use momentum construction and develop this question on decomposable ruled surfaces as an explicit model. At that time, Proposition \ref{prop:S1-invariant-psc-cone} was suggested by ChatGPT~5.5 as a key step of the counterexample when the author \textit{believed} the above question was always true. After that, especially after the variational framework of PSC K\"ahler metrics \cite{Sha2026Variational} was developed, the author realized that the existence of PSC K\"ahler metrics in a given class is related to algebraic degeneration \cite{Sha2026stability}, which may not be determined by intersection numbers alone. So the author studied the obstruction and related it to the slope stability for ruled surfaces. All mathematical statements in the main context were developed by the author. GPT~5.5 and 5.6-sol were used as an assisting tool to refine the writing and presentation, and to help verify the mathematical logic.

\subsection*{Acknowledgements} 
The author is sincerely grateful to Prof. Xiuxiong Chen for his continued encouragement and support. The author also thanks Longteng Chen, Song Sun, Jian Wang, Mingchen Xia, and Qi Yao for helpful discussions.

\section{Preliminaries}\label{sec:preliminaries}

\subsection{Geometry of ruled surfaces} \label{sec:ruled-prelim}

We recall some fundamental facts of ruled surfaces and keep our notations. Standard references for ruled surfaces are \cite[Chapter~V, Section~2]{Hartshorne1977} and \cite[Chapter~V]{BHPV2004}.

Let \(X=\Pj(E)\xrightarrow{\pi}\Sigma_g\) be a geometrically ruled surface, and let \(F\) and \(\zeta\) be the fibre class and the normalized tautological class introduced in Section~\ref{sec:introduction}. The intersection numbers are given by
\begin{equation}
\zeta^2=0, \qquad \zeta\cdot F=1, \qquad F^2=0.
\label{eq:ruled-intersections}
\end{equation}
The first Chern class is then given by
\begin{equation}
c_1(X) = 2\zeta+(2-2g)F.
\label{eq:ruled-c1}
\end{equation}
Consequently, if \(A=a\zeta+bF\), then
\begin{equation}
A^2=2ab, \qquad c_1(X)\cdot A = 2b+(2-2g)a.
\label{eq:ruled-total}
\end{equation}
Suppose that \(E\) is unstable. Then the Harder--Narasimhan filtration \cite{HarderNarasimhan1975} induces a short exact sequence
\[
0\longrightarrow L\longrightarrow E\longrightarrow Q \longrightarrow0,
\]
where \(\deg L=\mu_{\max}(E)\). The quotient \(E\twoheadrightarrow Q\) determines a section \(C\subset X\), called the \emph{Harder--Narasimhan section}. Thus \(C\) is the section associated with the quotient of \(E\) by its maximal-slope subbundle; see also \cite{LangeNarasimhan1983}. Its numerical class and self-intersection are
\begin{equation}
C \equiv \zeta-\frac{\Delta(E)}2F, \qquad C^2=-\Delta(E).
\label{eq:HN-section}
\end{equation}
In the decomposable case \(E=\mathcal O_{\Sigma_g}\oplus L\) with \(\deg L=-d<0\), one has \(d=\Delta(E)\). We denote the Harder--Narasimhan section by \(C_0\) induced, in the quotient convention, by \(\mathcal O_{\Sigma_g}\oplus L\rightarrow L\). It satisfies
\begin{equation} \label{eq:decomposable-section}
 C_0^2=-d, \qquad C_0=\zeta-\frac d2F.    
\end{equation}

For an arbitrary rank-two vector bundle \(E\), the K\"ahler cone of \(X=\Pj(E)\) is given by
\begin{equation}
A=a\zeta+bF\in\Kah_X \quad\Longleftrightarrow\quad a>0, \qquad b>\frac{\Delta(E)}2a.
\label{eq:kahler-cone}
\end{equation}
Indeed, this follows by taking the interior of the nef cone described by Miyaoka \cite{Miyaoka1987}; see also \cite[Lemma~2.1]{Fulger2011}. Combining \eqref{eq:ruled-total} and \eqref{eq:kahler-cone}, we obtain
\begin{equation}
\mathcal T_X^+ = \left\{ 2\pi(a\zeta+bF): a>0,\quad b> \max\left\{ g-1,\frac{\Delta(E)}2 \right\}a \right\}.
\label{eq:total-cone-ruled}
\end{equation}

\subsection{Openness and deformation of PSC K\"ahler classes}
\label{sec:psc-openness}

In this subsection, we show some basic facts about the PSC K\"ahler class. First, for any \(\alpha \in \cpsc\), there exists an open neighborhood \(U\) of \(\alpha\), such that \(\overline{U} \subset \cpsc\). In particular, the PSC K\"ahler cone is open.

\begin{lemma}
\label{lem:psc-openness}
For every compact K\"ahler manifold \(X\), the PSC K\"ahler cone \(\Kah_X^{\mathrm{psc}}\) is open in \(H^{1,1}(X,\mathbb R)\).
\end{lemma}
\begin{proof}
If \(\cpsc=\varnothing\), then there is nothing to prove. We assume \(\cpsc \neq \varnothing\) and fix a PSC K\"ahler metric \(\omega\) on \(X\). Choose smooth closed real \((1,1)\)-forms \(\theta_1,\ldots,\theta_m\) whose classes form a basis of \(H^{1,1}(X,\mathbb R)\). Define
\[
 \omega_{\boldsymbol s}:=\omega+\sum_{i=1}^m s_i\theta_i,\qquad \boldsymbol s=(s_1,\ldots,s_m)\in\mathbb R^m.
\]
Clearly, \(\omega_{\boldsymbol s}\) are closed and of type \((1,1)\) for all \(\boldsymbol s\in\mathbb R^m\) . For \(\boldsymbol s\) sufficiently close to \(0\), \(\omega_{\boldsymbol s} >0\) gives a K\"ahler metric on \(X\). Moreover, \(\omega_{\boldsymbol s}\rightarrow \omega\) smoothly as \(\boldsymbol{s} \rightarrow 0\). Since the scalar curvature depends continuously on the metric in the \(C^2\) topology, and \(\min_X S(\omega)>0\), we have \(S(\omega_{\boldsymbol s})>0\) for all sufficiently small \(\boldsymbol s\). Varying \(\boldsymbol{s}\) gives an open neighborhood of \([\omega]\) in \(H^{1,1}(X,\mathbb R)\) such that each class in this neighborhood still contains a PSC K\"ahler metric. This completes the proof.
\end{proof}

We next consider the deformation of PSC K\"ahler metrics.
\begin{lemma}
\label{lem:psc-deformation}
Let \(\mathbb D\subset\mathbb C\) be an open disc centred at \(0\), and let \(\varpi:\mathcal X\to\mathbb D\) be a proper holomorphic submersion, with fibres \(X_t=\varpi^{-1}(t)\). Let \(\Theta\) be a smooth closed real \((1,1)\)-form on \(\mathcal X\). Suppose that \(X_0\) carries a PSC K\"ahler metric \(\omega_0\) satisfying
\[
 [\omega_0]=[\Theta|_{X_0}] \qquad\text{in }H^2(X_0,\mathbb R).
\]
Then, after shrinking \(\mathbb D\), \(\omega_0\) extends to a smooth family of PSC K\"ahler metrics \(\omega_t\) on \(X_t\) such that
\[
 [\omega_t]=[\Theta|_{X_t}]  \qquad\text{in }H^2(X_t,\mathbb R)  \quad\text{for every }t\in\mathbb D.
\]
\end{lemma}

\begin{proof}
Since \(X_0\) is compact K\"ahler and \(\omega_0-\Theta|_{X_0}\) is an exact real \((1,1)\)-form, by \(\partial\bar\partial\)-lemma, there is a smooth real function \(\varphi_0\) on \(X_0\) such that
\[
 \omega_0=\Theta|_{X_0}+\ddbar\varphi_0.
\]

By Ehresmann's theorem \cite[Theorem~9.3]{Voisin2002}, after shrinking \(\mathbb D\) and restricting \(\mathcal X\) accordingly, there is a smooth trivialization
\[
 \Psi:X_0\times\mathbb D\longrightarrow\mathcal X, \qquad \varpi(\Psi(x,t))=t,\qquad \Psi(x,0)=x.
\]
Define a smooth real function \(\widetilde \varphi\) on \(\mathcal X\) by \(\widetilde \varphi(\Psi(x,t))=\varphi_0(x)\), and define
\[
 \omega_t:=(\Theta+\ddbar\widetilde \varphi)|_{X_t}. 
\]
In particular, \(\omega_t\) are closed, real, and of type \((1,1)\), have the required cohomology classes, and agree with \(\omega_0\) at \(t=0\).

Write \(\Psi_t=\Psi(\,\cdot\,,t)\), and use \(\Psi_t\) to pull the complex structure of \(X_t\) back to a complex structure \(J_t\) on the fixed smooth manifold \(X_0\). Both \(J_t\) and \(\Psi_t^*\omega_t\) vary smoothly in \(t\), and \(\Psi_t^*\omega_t\) converges smoothly to \(\omega_0\) as \(t\to0\). Moreover, \(\Psi_t^*\omega_t\) is positive with respect to \(J_t\) for all sufficiently small \(t\). Thus \(\omega_t\) is K\"ahler. Since the forms and the complex structures both vary smoothly, the same argument as in the proof of Lemma~\ref{lem:psc-openness} gives \(S(\omega_t)\circ\Psi_t\longrightarrow S(\omega_0)\) uniformly on \(X_0\). Since \(\min_{X_0}S(\omega_0)>0\), the metrics \(\omega_t\) have positive scalar curvature after a further shrinking of \(\mathbb D\), if necessary.
\end{proof}

\subsection{The non-Archimedean \(\Fcal\)-functional}\label{sec:NA}

We briefly recall the variational characterization of PSC K\"ahler metrics together with its non-Archimedean formulation. 

Let \(X\) be a smooth projective surface and let \(A\) be a \(\mathbb Q\)-ample class satisfying \(c_1(X)\cdot A>0\). Write \(\alpha=2\pi A\), \(V_\alpha=\alpha^2\), and fix a reference K\"ahler form \(\omega\in\alpha\).  Since
\[
\overline S_\alpha=\frac{4\pi c_1(X)\cdot\alpha}{\alpha^2}=\frac{2c_1(X)\cdot A}{A^2}>0,
\]
the measure \(\Omega=\overline S_\alpha\,\omega^2\) belongs to the admissible measure considered in \cite{Sha2026Variational}. For \(\varphi\in\Hpot_\omega\), the corresponding prescribed scalar-curvature measure functional is
\begin{align*}
\mathcal F_{\omega,\Omega}(\varphi) &= \mathcal M_\omega(\varphi) - \overline S_\alpha E_\omega(\varphi)+\int_X \varphi\,\Omega  \\
&= \mathcal M_\omega(\varphi) +  \overline S_\alpha  \left(\int_X \varphi\,\omega^2-E_\omega(\varphi)  \right),
\end{align*}
where \(\mathcal M_\omega\) is the Mabuchi \(K\)-energy and \(E_\omega\) is the Aubin-Yau energy. After dividing both sides by \(2V_\alpha\), the second term is the Aubin \(J\)-functional, and we denote these normalized functionals by
\[
 M(\varphi):= \frac{\mathcal M_\omega(\varphi)}{2V_\alpha},  \qquad J(\varphi) :=  \frac{1}{2V_\alpha}  \left( \int_X \varphi\,\omega^2-E_\omega(\varphi)  \right), \qquad  \Fcal(\varphi) :=  \frac{\mathcal F_{\omega,\Omega}(\varphi)}{2V_\alpha}.
\]
Thus we can simply write the \(\mathcal F\)-functional as
\begin{equation} \label{eq:F-MJ}
\Fcal = M+\overline S_\alpha J.
\end{equation}

Since \(A\) is rational, we identify it with the numerical class of an ample \(\mathbb Q\)-line bundle on \(X\). Recall that a normal ample test configuration for \((X,A)\) consists of a normal variety \(\mathcal X\), a flat projective morphism \(\varpi:\mathcal X\rightarrow\mathbb A^1\), a \(\mathbb C^*\)-action on \(\mathcal X\) lifting the standard action on \(\mathbb A^1\), and a \(\mathbb C^*\)-linearized relatively ample \(\mathbb Q\)-line bundle \(\mathcal A\), together with an equivariant identification
\[
(\mathcal X,\mathcal A)|_{\mathbb C^*}\simeq (X,A)\times\mathbb C^*;
\]
see \cite[Definitions~2.1 and~2.2]{BHJ2017}. For such a test configuration, we define
\begin{equation}  \label{eq:def-FNA}
\Fcal^{\NA}(\mathcal X,\mathcal A) := M^{\NA}(\mathcal X,\mathcal A)  + \overline S_\alpha J^{\NA}(\mathcal X,\mathcal A),
\end{equation}
where \(M^{\NA}\) and \(J^{\NA}\) are the non-Archimedean Mabuchi \(K\)-energy and Aubin \(J\)-functionals of \cite[Definitions~7.13 and~7.6]{BHJ2017}.

\begin{lemma}\label{lem:PSC-NA-slope}
If \(\alpha\) contains a K\"ahler metric of positive scalar curvature, then there exists a constant \(\delta>0\) such that
\begin{equation}        \label{eq:uniform-NA}
        \Fcal^{\NA}(\mathcal X,\mathcal A) \geq \delta J^{\NA}(\mathcal X,\mathcal A)
\end{equation}
for every normal ample test configuration \((\mathcal X,\mathcal A)\) for \((X,A)\).
\end{lemma}

\begin{proof}
Since \(\alpha\) contains a PSC K\"ahler metric, by \cite[Theorem~A]{Sha2026Variational}, the \(\Fcal\)-functional is \(d_1\)-coercive on \(\mathcal E^1_0(X,\omega)\). Thus there exist constants \(\varepsilon>0\) and \(C>0\) such that
\[
    \Fcal(\varphi) \geq  \varepsilon d_1(0,\varphi)-C.
\]
On \(\mathcal E^1_0(X,\omega)\), \cite[Eq.~(62)]{Darvas2015} gives a constant \(C_J>0\) such that
\[
   d_1(0,\varphi) \geq  C_J^{-1}J(\varphi).
\]
Consequently,
\begin{equation}\label{eq:F-J-coercivity}
   \Fcal(\varphi)  \geq \varepsilon C_J^{-1}J(\varphi)-C.
\end{equation}

Let \((\mathcal X,\mathcal A)\) be a normal ample test configuration, and let \(\varphi_s\) be the ray induced by a smooth strictly positive \(S^1\)-invariant metric on \(\mathcal A\) near the central fibre. Define \(\widehat\varphi_s:=\varphi_s -E_\omega(\varphi_s)/V_\alpha\). Then \(\widehat\varphi_s\in\mathcal E^1_0(X,\omega)\).  Since \(\Fcal\) and \(J\) are invariant under addition of constants,
\[
        \Fcal(\widehat\varphi_s)=\Fcal(\varphi_s), \qquad  J(\widehat\varphi_s)=J(\varphi_s).
\]
Applying \eqref{eq:F-J-coercivity}, dividing by \(s\), and letting \(s\to+\infty\), the slope formulas \cite[Theorem~A]{BHJ2019} give
\[
   \Fcal^{\NA}(\mathcal X,\mathcal A)  \geq  \varepsilon C_J^{-1}   J^{\NA}(\mathcal X,\mathcal A).
\]
This proves the assertion with \(\delta=\varepsilon C_J^{-1}\).
\end{proof}

Let \(\varpi:(\mathcal X,\mathcal A)\rightarrow\mathbb A^1\) be a normal ample test configuration for \((X,A)\). The compactification \(\overline\varpi: (\overline{\mathcal X},\overline{\mathcal A}) \rightarrow\mathbb P^1\) is obtained by gluing \((\mathcal X,\mathcal A)\) to \((X,A)\times(\mathbb P^1\setminus\{0\})\) along their common restriction over \(\mathbb C^*\); see \cite[Definition~2.7]{BHJ2017}.  Thus \(\overline{\mathcal X}\) is a normal threefold, \(\overline{\mathcal A}\) is the resulting relatively ample \(\mathbb Q\)-line bundle, and
\[
 (\overline{\mathcal X},\overline{\mathcal A})   |_{\mathbb P^1\setminus\{0\}}  \simeq  (X,A)\times(\mathbb P^1\setminus\{0\}).
\]

Choose a common \(\mathbb C^*\)-equivariant resolution of the birational map induced by the above identification,
\[
 \mu:\mathcal W\longrightarrow\overline{\mathcal X}, \qquad \rho:\mathcal W\longrightarrow X\times\mathbb P^1.
\]
Let \(\pr_X:X\times\mathbb P^1\to X\) be the projection, and write \(H=\rho^* \pr_X^*A\). All mixed intersections involving \(H\) are computed on \(\mathcal W\), with classes from \(\overline{\mathcal X}\) pulled back by \(\mu\). In particular, these intersection numbers are independent of the choice of common resolution; see \cite[Section~6.6]{BHJ2017}. If the birational map extends to a morphism \(\overline{\mathcal X}\to X\times\mathbb P^1\), we say that \(\overline{\mathcal X}\) dominates the trivial compactification, and intersection numbers can be computed directly on \(\overline{\mathcal X}\).

Let \(\mathcal X_0=\overline\varpi^{-1}(0)\) denote the scheme-theoretic central fibre. Replacing \(\overline{\mathcal A}\) by \(\overline{\mathcal A}+s\mathcal X_0\), for \(s\in\mathbb Q\), corresponds to translating the associated non-Archimedean metric. Since \(J^{\NA}\) and \(\Fcal^{\NA}\) are translation invariant, and since \(\mu^*\mathcal X_0\cdot H^2=A^2\), we may translate \(\overline{\mathcal A}\) so that
\begin{equation}        \label{eq:NA-normalization}
        \mu^*\overline{\mathcal A}\cdot H^2=0.
\end{equation}
Since \(\dim X=2\), by \cite[Definition~6.11 and 7.2]{BHJ2017}, the non-Archimedean Aubin-Yau energy is defined by
\[
        E^{\NA}(\mathcal X,\mathcal A) := \frac{\overline{\mathcal A}^{\,3}}{3A^2}.
\]
Moreover, \cite[Definitions~7.2 and~7.6]{BHJ2017} gives
\[
        J^{\NA}(\mathcal X,\mathcal A) =  \frac{\overline{\mathcal A}\cdot H^2}{A^2}  - E^{\NA}(\mathcal X,\mathcal A).
\]
Define the relative log canonical divisor by
\begin{equation} \label{eq:def-relative-log-canonical}
K^{\log}_{\overline{\mathcal X}/\mathbb P^1}:= K_{\overline{\mathcal X}/\mathbb P^1}- \bigl( \mathcal X_0-\mathcal X_{0,\mathrm{red}} \bigr).
\end{equation}
By \cite[Definition~7.13]{BHJ2017},
\[
M^{\NA}(\mathcal X,\mathcal A) =\frac{K^{\log}_{\overline{\mathcal X}/\mathbb P^1} \cdot\overline{\mathcal A}^{\,2}}{A^2}+ \overline S_\alpha  E^{\NA}(\mathcal X,\mathcal A).
\]
Combining these formulas with \eqref{eq:def-FNA} and \eqref{eq:NA-normalization}, we obtain
\begin{align}  \label{eq:JNA-FNA}
\Fcal^{\NA}(\mathcal X,\mathcal A) = \frac{K^{\log}_{\overline{\mathcal X}/\mathbb P^1}\cdot\overline{\mathcal A}^{\,2}}{A^2}, 
\qquad J^{\NA}(\mathcal X,\mathcal A)=-\frac{\overline{\mathcal A}^{\,3}}{3A^2}.
\end{align}
If the central fibre is reduced scheme-theoretically, then \(\mathcal X_0=\mathcal X_{0,\mathrm{red}}\), and hence \(K^{\log}_{\overline{\mathcal X}/\mathbb P^1}=K_{\overline{\mathcal X}/\mathbb P^1}\).

\subsection{Intersection formulas on total spaces} \label{sec:basic-intersections}

In this part, we present the basic numerical classes and compute their intersections on the compactified total spaces that we will use later.

\vspace{2mm}

Let \(X=\Pj(E)\to\Sigma_g\), where \(E\) is unstable. Consider the Harder--Narasimhan filtration \(L \subset E\). Write \(Q= E/L\). Then we obtain a short exact sequence
\[
0\longrightarrow L\longrightarrow E\longrightarrow Q \longrightarrow0, \qquad \deg L-\deg Q=\Delta(E)>0.
\] 
Let \(C=\Pj(Q)\subset X\) be the Harder--Narasimhan section. We denote by \(\pr_M\) the projection of the Cartesian product onto the \(M\)-factor. Let \(\overline{\mathcal E}\) be the extension
\[
 0\longrightarrow \mathcal L \longrightarrow\overline{\mathcal E} \longrightarrow \mathcal{Q} \longrightarrow0, \qquad \mathcal L :=\pr_{\Sigma_g}^*L\otimes \pr_{\Pj^1}^*\mathcal O_{\Pj^1}(1),~ \mathcal{Q} := \pr_{\Sigma_g}^*Q
\]
obtained by multiplying the extension class of \(E\) by a nonzero section \(s\in H^0(\Pj^1,\mathcal O_{\Pj^1}(1))\) vanishing at \(0\) (see \cite[Remark~5.14]{RossThomas2006}). Write \(\overline\Pi: \overline{\mathcal X}=\Pj(\overline{\mathcal E}) \rightarrow\Sigma_g\times\Pj^1,~\overline\varpi=\pr_{\Pj^1}\circ\overline\Pi\), and write \(\mathcal X\), \(\Pi\), and \(\varpi\) for their restrictions over \(\mathbb A^1\). In particular, \(\varpi: \mathcal X \rightarrow \mathbb A^1\) is a smooth family, with
\[
 \mathcal X_0=\Pj(L\oplus Q),\qquad \mathcal X_t\simeq X\quad\text{for }t\ne0.
\]

On \(\overline{\mathcal X}\), define
\begin{equation} \label{eq:class-total-space}
 \bar\zeta  =c_1\bigl(\mathcal O_{\overline{\mathcal X}}(1)\bigr) -\frac12\overline\Pi^*c_1(\overline{\mathcal E}), \qquad \bar F=\overline\Pi^*\pr_{\Sigma_g}^*[\mathrm{pt}].
\end{equation}
The restriction of \(\bar\zeta\) to each fibre is the normalized tautological class on that fibre. We also write \(\mathcal X_0=\overline\varpi^*[0]\) for the divisor class of the central fibre. In particular, \(\bar\zeta\), \( \bar F\), and \(\mathcal X_0\) form a basis of \(N^1(\overline{\mathcal X})_{\mathbb R}\).

By the Whitney formula, we have
\[
1 +  c_1(\overline{\mathcal E}) +  c_2(\overline{\mathcal E}) = (1+ c_1(\mathcal L)) (1+ c_1(\mathcal Q)).
\]
Therefore, we obtain
\[ 
c_1(\overline{\mathcal E}) = c_1(\mathcal L) + c_1(\mathcal Q) = (\deg E) \,\pr_{\Sigma_g}^*[\mathrm{pt}] +  \pr_{\Pj^1}^*[0],
\]
and
\[
c_2(\overline{\mathcal E})= c_1(\mathcal L) \cdot c_1(\mathcal Q) = (\deg Q)\,\pr_{\Sigma_g}^*[\mathrm{pt}] \cdot\pr_{\Pj^1}^*[0].
\]
Applying the Whitney formula to the short exact sequence 
\[
0 \longrightarrow\overline{\Pi}^*\det  \overline{ \mathcal E}\otimes \mathcal{O}_{\overline{\mathcal X}}(-1) \longrightarrow \overline{\Pi}^* \overline{ \mathcal E} \longrightarrow  \mathcal{O}_{\overline{\mathcal X}}(1) \longrightarrow 0,
\]
we then obtain
\[
 c_1\bigl(\mathcal O_{\overline{\mathcal X}}(1)\bigr)^2-c_1(\overline{\mathcal E})\cdot c_1\bigl(\mathcal O_{\overline{\mathcal X}}(1)\bigr) +c_2(\overline{\mathcal E})=0.
\]
Consequently, we have
\[
 \bar\zeta^2=\frac14c_1(\overline{\mathcal E})^2 -c_2(\overline{\mathcal E}) =\left(\frac{\deg E}{2}-\deg Q\right)\bar F\cdot\mathcal X_0 =\frac{\Delta(E)}2\bar F\cdot\mathcal X_0.
\]
Since \(\bar F^2=\mathcal X_0^2=0\) and \(\bar F \cap\mathcal X_0\simeq \mathbb P^1\), the only non-zero intersection numbers between the base classes are
\begin{equation}  \label{eq:basic-extension-triple-intersections}
 \bar\zeta^3=\frac{\Delta(E)}2,\qquad \bar\zeta\cdot \bar F \cdot\mathcal X_0=1.
\end{equation}

\vspace{2mm}

Next, we consider the blowup \(p:X=\operatorname{Bl}_{q_1,\ldots,q_N}Y \rightarrow Y=\Pj(E)\), where the points \(q_i\) lie outside the Harder--Narasimhan section \(C\) and in pairwise distinct fibres of \(\pi_Y:Y\to\Sigma_g\). Let \(D_i\subset X\) be the exceptional divisors. Apply the preceding construction to \(Y\), and denote its compactified family and its projections by
\[
 \overline\Pi_Y:\overline{\mathcal Y}=\Pj(\overline{\mathcal E})\longrightarrow\Sigma_g\times\Pj^1,\qquad \overline\varpi_Y=\pr_{\Pj^1}\circ\overline\Pi_Y,
\]
where we identify \(Y\) with \(\mathcal Y_1\). The \(C^*\)-equivariant action induced by scaling the extension class gives an orbit section through \(q_i\) that extends to a section \(s_i:\mathbb A^1\to\mathcal Y\), and then to a section \(\overline s_i:\Pj^1\to\overline{\mathcal Y}\). Write \(\mathcal Z_i\) for its image, using the same notation before and after compactification. These sections are disjoint because their projections \(x_i:=\pi_Y(q_i)\) are distinct.

Blowing up the sections gives \(\overline P:\overline{\mathcal X} =\operatorname{Bl}_{\coprod_i\mathcal Z_i}  \overline{\mathcal Y} \rightarrow\overline{\mathcal Y}\), with exceptional divisors \(\mathcal D_i\). Let \(P\) be its
restriction over \(\mathbb A^1\), and write the induced parameter map by \(\overline\varpi_X=\overline\varpi_Y\circ\overline P\). In particular, the family is smooth with general fibre \(X\), and \(\mathcal D_i\) restricts to \(D_i\) on each general fibre. Moreover, \(\mathcal X_0=\overline P^{\,*}\mathcal Y_0\). Suppressing pullbacks of \(\bar\zeta\) and \(\bar F\), the basic numerical divisor classes on \(\overline{\mathcal X}\) are \(\bar\zeta\),\(\bar F\), \(\mathcal X_0\) and \(\mathcal D_1,\ldots,\mathcal D_N\). Since \eqref{eq:basic-extension-triple-intersections} remains valid for \(\bar\zeta\), \(\bar F\) and \(\mathcal X_0\), we only need to compute intersection numbers concerning \(\mathcal D_i\).

For divisor classes \(U,V \in \{\bar\zeta,\bar F, \mathcal Y_0\}\) on \(\overline{\mathcal Y}\), by \cite[Equation~(6.11)]{Aluffi} and \cite[Lemma~23]{AraujoCastravet2013}, we have
\begin{equation} \label{eq:basic-curve-blowup}
 \overline P^{\,*}U\cdot\overline P^{\,*}V\cdot\mathcal D_i=0,\qquad \overline P^{\,*}U\cdot\mathcal D_i^2=-U\cdot\mathcal Z_i,\qquad \mathcal D_i^3=-\deg N_{\mathcal Z_i/\overline{\mathcal Y}}.
\end{equation}
Let \(S_i=\overline{\Pi}_Y^{-1}(\{x_i\}\times\mathbb P^1) \subset \overline{\mathcal Y}\). Then \([S_i]\equiv \bar F\) and \(\mathcal Z_i\subset S_i\). By \cite[Chapter~V, Proposition~2.6]{Hartshorne1977}, we have \(\mathcal{O}_{S_i}(\mathcal{Z}_i)\simeq \mathcal{O}_{S_i}(1)\). Therefore, we have
\[
[\mathcal{Z}_i]=\bar F \cdot \bar \zeta +\frac{1}{2} \bar F \cdot \mathcal{Y}_0.
\]
Consequently, we obtain
\begin{equation}\label{eq:intersect-orbit-sec}
\bar \zeta \cdot \mathcal Z_i = \frac{1}{2},\qquad \bar F \cdot \mathcal Z_i = 0, \qquad \mathcal Y_0 \cdot \mathcal Z_i = 1.
\end{equation}
Since 
\[
\deg N_{\mathcal Z_i/\overline{\mathcal Y}} = c_1(T_{\overline{\mathcal Y}/\mathbb P^1}) \cdot \mathcal Z_i= 2\bar \zeta \cdot \mathcal Z_i + (2-2g)\bar F \cdot \mathcal Z_i= 1,
\]
we have \(\mathcal D_i^3=-1\). Substituting into \eqref{eq:basic-curve-blowup}, we obtain
\begin{equation} \label{eq:basic-orbit-exceptional-intersections}
 \bar\zeta \cdot \mathcal D_i^2=-\frac12,\qquad \bar F \cdot \mathcal D_i^2=0,\qquad \mathcal X_0 \cdot \mathcal D_i^2=-1,\qquad \mathcal D_i^3=-1.
\end{equation}

\section{PSC K\"ahler cones on compact minimal surfaces}
\label{sec:HN-obstruction}

In this section, we prove Theorem~\ref{thm:main-minimal}. We first derive the necessary inequality for unstable ruled surfaces using the test configuration associated with the Harder–Narasimhan filtration. We then determine the PSC K\"ahler cone in different cases.

\subsection{Obstruction from the Harder--Narasimhan section}

Let \(X=\Pj(E)\rightarrow\Sigma_g\), for \(g \ge 1\), where \(E\) is unstable, and let \(C\) be the Harder--Narasimhan section.

\begin{proposition} \label{prop:HN-obstruction}
If the K\"ahler class \(2\pi A=2\pi(a\zeta+bF)\) contains a PSC K\"ahler metric, then
\begin{equation}
        2b>(\Delta(E)+2g-2)a.
        \label{eq:HN-obstruction}
\end{equation}
\end{proposition}

\begin{proof}
Assume first that \(A\) is rational.  Consider the compactified family \(\overline\varpi: \overline{\mathcal{X}} \to \Pj^1\) constructed in Section \ref{sec:basic-intersections}, with central fibre \(\mathcal X_0 = \Pj(L\oplus Q)\) and \(\mathcal{X}_t \simeq X\) for \(t \neq 0\). Let 
\[
\overline{\mathcal{A}}= a \bar \zeta + b \bar F -\frac{a}{2}\mathcal{X}_0, \qquad \mathcal{A} = \overline{\mathcal{A}}|_{\mathcal X}.
\]
Since \(A\) is K\"ahler, \(\overline{\mathcal{A}}\) restricts to an ample class on each fibre and is relative ample. Together with the \(C^*\)-equivariant action induced by scaling the extension class during the construction, we see \((\mathcal X, \mathcal A)\) is a normal ample test configuration for \((X,A)\) with smooth central fibre.

We next verify \eqref{eq:NA-normalization}. Let \(\mathcal Z\) be an orbit section of any general point in \(X \setminus C\). By \eqref{eq:intersect-orbit-sec}, we have \(\overline{\mathcal{A}} \cdot \mathcal Z = 0\). Since \(A\) is ample, choose a positive integer \(m\) such that \(mA\) is very ample, and take two general divisors \(G_1,G_2\in|mA|\) intersecting transversely at points \(q_1,\ldots,q_N\in X\setminus C\), where \(N=m^2A^2\). On \(X\times\mathbb P^1\), the pullbacks of \(G_1,G_2\) intersect along the curves \(\{q_j\}\times\mathbb P^1\). Let \(\mu:\mathcal W\to\overline{\mathcal X}\), \(\rho:\mathcal W\to X\times\mathbb P^1\) be the common resolution introduced in Section \ref{sec:NA}. By choosing \(G_1\) and \(G_2\) sufficiently generally, we may also assume that \(\rho\) is an isomorphism over each of these curves. Denote their lifts to \(\mathcal W\) by \(\widetilde{\mathcal Z}_j\). Each \(\widetilde{\mathcal Z}_j\) maps birationally under \(\mu\) onto the compactified orbit section \(\mathcal Z_j\) through \(q_j\). Thus, with \(H=\rho^*\operatorname{pr}_X^*A\), we have
\[
m^2H^2\equiv\sum_{j=1}^N[\widetilde{\mathcal Z}_j], \qquad \mu_*[\widetilde{\mathcal Z}_j]=[\mathcal Z_j].
\]
Therefore,
\[
\overline{\mathcal A}\cdot H^2=\frac1{m^2}\sum_{j=1}^N \mu^*\overline{\mathcal A}\cdot\widetilde{\mathcal Z}_j=\frac1{m^2}\sum_{j=1}^N \overline{\mathcal A}\cdot\mathcal Z_j =0,
\]
which proves \eqref{eq:NA-normalization}.

\vspace{2mm}

Since \(K_{\overline{{\mathcal X}}/\Pj^1}=-2\bar \zeta  + (2g-2)\bar F\), by \eqref{eq:basic-extension-triple-intersections}, we have
\begin{align}  \label{eq:HN-A-cube-F-numerator}
K_{\overline{\mathcal X}/\Pj^1}\cdot\overline{\mathcal A}^{\,2}= (2-2g-\Delta(E))a^2 + 2ab, \qquad \overline{\mathcal A}^{\,3} =\frac{\Delta(E)}{2} a^3-3a^2b.
\end{align}
Plugging into \eqref{eq:JNA-FNA}, we then obtain
\begin{align}  \label{eq:HN-JNA-FNA}
\Fcal^{\NA}(\mathcal X,\mathcal A)=\frac{(2-2g-\Delta(E))a^2 + 2ab}{A^2},\qquad J^{\NA}(\mathcal X,\mathcal A) =  \frac{a^2}{3A^2}\left( 3b-\frac{\Delta(E)}{2} a  \right)>0.
\end{align}

If \(2\pi A\) contains a PSC K\"ahler metric, by Lemma~\ref{lem:PSC-NA-slope}, we have \(\Fcal^{\NA} \ge \delta J^{\NA}\) for some \(\delta >0\). It follows from \eqref{eq:HN-JNA-FNA} that
\[
2b>(\Delta(E)+2g-2)a,
\]
which proves \eqref{eq:HN-obstruction} for rational classes.

Now let \(A=a\zeta+bF\) be real and suppose that \(2\pi A\) contains a PSC K\"ahler metric. By Lemma~\ref{lem:psc-openness}, for \(\eps>0\) small enough, \(2\pi (A-\eps F)\) also contains a PSC K\"ahler metric. Approximating \(A-\eps F\) by rational classes and apply \eqref{eq:HN-obstruction} for rational classes, we obtain
\[
2(b -\eps) - (\Delta(E)+2g-2)a \ge 0.
\]
Since \(\eps>0\) is arbitrarily small, we complete the proof.
\end{proof}

\subsection{Proof of Theorem \ref{thm:main-minimal}}
\label{sec:minimal-proof}

In this subsection, we give the proof of Theorem \ref{thm:main-minimal}. We begin with minimal rational surfaces.

\begin{lemma}
\label{lem:rational-minimal}
If \(X=\Pj^2\) or \(X=\mathbb F_n\) with \(n=0\) or \(n\ge2\), then
\[
\Kah_X^{\mathrm{psc}} = \mathcal T_X^+ = \Kah_X.
\]
\end{lemma}

\begin{proof}
The projective plane and the Hirzebruch surfaces are compact smooth toric K\"ahler manifolds.  By \cite[Theorem~C]{Sha2026Variational}, every K\"ahler class on a compact smooth toric manifold contains a torus-invariant PSC K\"ahler metric. Hence \(\Kah_X^{\mathrm{psc}}=\Kah_X\), which completes the proof.
\end{proof}

We next turn to non-rational minimal ruled surfaces.  We first treat the case when \(E\) is semistable.

\begin{lemma}
\label{lem:semistable-cone}
Let \(X=\Pj(E)\to\Sigma_g\), where \(g\ge1\) and \(E\) is
slope-semistable.  Then
\begin{equation}
\Kah_X^{\mathrm{psc}} = \mathcal T_X^+ = \left\{2\pi(a\zeta+bF): a>0,\quad b>(g-1)a \right\}.
\label{eq:semistable-PSC-cone}
\end{equation}
\end{lemma}

\begin{proof}
Suppose first that \(E\) is polystable. By \cite[Theorem~2]{ApostolovTonnesen2006}, \(\Pj(E)\) admits a locally symmetric cscK metric in every K\"ahler class.  If the class belongs to \(\mathcal T_X^+\), the constant scalar curvature of this metric is positive. Thus \(\mathcal T_X^+\subseteq\Kah_X^{\mathrm{psc}}\).

Fix a K\"ahler class \(\alpha =2\pi(a\zeta+bF )\in \ctpsc\) on \(X\). Suppose next that \(E\) is semistable but not polystable. Then \(E\) is strictly semistable and admits the Jordan--H\"older filtration \(L \subset E\). Write \(Q = E/L\), we then obtain a short exact sequence:
\[
0\longrightarrow L\longrightarrow E\longrightarrow Q \longrightarrow0, \qquad \deg L=\deg Q=\frac{\deg E}{2}.
\]
Let \(\xi\in\operatorname{Ext}^1(Q,L)\) be the class of this extension, and let \(\mathbb D\subset\mathbb C\) be an open disc centred at \(0\). Scaling \(\xi\) to \(t\xi\) gives a holomorphic vector bundle \(\mathcal E\to\Sigma_g\times\mathbb D\), whose restriction \(E_t:=\mathcal E|_{\Sigma_g\times\{t\}}\) is the extension of \(Q\) by \(L\) with class \(t\xi\). Hence
\[
E_0=L\oplus Q, \qquad E_t\simeq E \quad\text{for }t\ne0.
\]
Let \(\varpi:\mathcal X=\Pj(\mathcal E)\to\mathbb D\) be the corresponding smooth proper holomorphic family, with fibres \(X_t=\Pj(E_t)\). Denote by \(\zeta_t\) and \(F_t\) the normalized tautological and fibre classes on \(X_t\). By \eqref{eq:class-total-space}, we can choose a smooth closed real \((1,1)\)-form \(\Theta\) representing \(2\pi(a\bar\zeta+b\bar F)\). Since \(\bar\zeta\) and \(\bar F\) restrict to \(\zeta_t\) and \(F_t\) on each \(X_t\), respectively, we have
\[
 [\Theta|_{X_t}]=\alpha_t:=2\pi(a\zeta_t+bF_t).
\]
Indeed, \(\alpha_0\in\mathcal T_{X_0}^+\). Since \(E_0=L\oplus Q\) is polystable, there is a positive cscK metric \(\omega_0\in\alpha_0\). By Lemma~\ref{lem:psc-deformation}, after shrinking \(\mathbb D\), \(\omega_0\) extends to a smooth family of PSC K\"ahler metrics \(\omega_t\in\alpha_t\) on \(X_t\). For \(t\ne0\), the bundle isomorphism \(E_t\simeq E\) over \(\Sigma_g\) induces a biholomorphism \(X_t\simeq X\) identifying \(\zeta_t,F_t\) with \(\zeta,F\). Thus \(\alpha\) is a PSC K\"ahler class on \(X\). Finally, \eqref{eq:total-cone-ruled} with \(\Delta(E)=0\) gives \eqref{eq:semistable-PSC-cone}.
\end{proof}

When \(E\) is unstable and \(g \ge 1\), Proposition~\ref{prop:HN-obstruction} already gives the necessary direction. It remains to establish the converse. We use the momentum construction on the split bundle associated with the Harder--Narasimhan filtration, given in Appendix~\ref{app:decomposable-model},

\begin{lemma}
\label{lem:unstable-cone}
Let \(X=\Pj(E)\to\Sigma_g\), where \(g\ge1\) and \(E\) is unstable. Then
\begin{equation}
\Kah_X^{\mathrm{psc}} = \left\{ 2\pi(a\zeta+bF): a>0,\quad b>\left(g-1+\frac{\Delta(E)}{2}\right)a \right\}.
\label{eq:unstable-PSC-cone}
\end{equation}
\end{lemma}

\begin{proof}
Let \(\alpha =2\pi(a\zeta+bF )\) satisfy \(a>0\) and \(b>\left(g-1+\frac{\Delta(E)}2\right)a\). Consider the Harder--Narasimhan filtration \(L \subset E\). Write \(Q= E/L\), we then obtain a short exact sequence:
\[
0\longrightarrow L\longrightarrow E\longrightarrow Q \longrightarrow0, \qquad \deg L-\deg Q=\Delta(E)>0.
\]
Scaling the extension class gives a holomorphic vector bundle \(\mathcal E\to\Sigma_g\times\mathbb D\) with
\[
E_0=L\oplus Q, \qquad E_t:=\mathcal E|_{\Sigma_g\times\{t\}}\simeq E \quad\text{for }t\ne0.
\]
After twisting by \(L^{-1}\), the central ruled surface becomes
\[
\Pj(E_0) \simeq \Pj(\mathcal O_{\Sigma_g}\oplus M), \qquad M=Q\otimes L^{-1}, \qquad \deg M=-\Delta(E).
\]

Let \(\varpi:\mathcal X=\Pj(\mathcal E)\to\mathbb D\) be the corresponding smooth proper holomorphic family, with fibres \(X_t=\Pj(E_t)\). Denote by \(\zeta_t\) and \(F_t\) the normalized tautological and fibre classes on \(X_t\). On the decomposable central fibre, in the notation of Appendix~\ref{app:decomposable-model}, take \(m=a\) and \(\lambda=\frac{1}{\Delta(E)}\left(b-\frac{\Delta(E)}{2}a\right)\). Lemma~\ref{lem:admissible-positive-scalar-curvature} therefore gives a PSC K\"ahler metric in the central class \(2\pi(a\zeta_0+bF_0)\). By Lemma~\ref{lem:psc-deformation}, after shrinking \(\mathbb D\), \(\omega_0\) extends to a smooth family of PSC K\"ahler metrics \(\omega_t\in\alpha_t\) on \(X_t\). For \(t\ne0\), the isomorphism \(E_t\simeq E\) over \(\Sigma_g\) identifies \(\alpha_t\) with \(\alpha\), as in the proof of Lemma~\ref{lem:semistable-cone}. Thus \(\alpha\) is a PSC K\"ahler class on \(X\). Combining with Proposition~\ref{prop:HN-obstruction} proves \eqref{eq:unstable-PSC-cone}.
\end{proof}

We can now complete the proof for minimal surfaces.

\begin{proof}[Proof of Theorem~\ref{thm:main-minimal}]
A minimal compact K\"ahler surface of Kodaira dimension \(-\infty\) is either \(\Pj^2\), a minimal Hirzebruch surface \(\mathbb F_n\) with \(n=0\) or \(n\ge2\), or a geometrically ruled surface over a curve of
positive genus; see \cite[Chapter~VI]{BHPV2004}. Lemma~\ref{lem:rational-minimal} proves the assertion (i).
Now let \(X=\Pj(E)\rightarrow\Sigma_g\) with \(g\ge1\). Lemma~\ref{lem:semistable-cone} computes the PSC K\"ahler cone when \(\Delta(E)=0\), while Lemma~\ref{lem:unstable-cone} computes it when \(\Delta(E)>0\). Combining with \eqref{eq:total-cone-ruled}, we prove the assertion (ii).
\end{proof}

\section{Birational obstructions and descent} \label{sec:birational-propagation}

In this section, we study the behavior of the PSC K\"ahler cone under birational morphisms.

\subsection{Proof of Theorem \ref{thm:main-birational-obstruction}}

We begin with a general propagation principle. A test configuration with strictly negative \(\Fcal^{\NA}\) can be lifted through point blowups while preserving this strict inequality, provided the blow-up parameters are chosen sufficiently small.

\begin{proposition} \label{prop:birational-propagation}
Let \(Y\) be a smooth projective surface. Assume that there exist a \(\mathbb Q\)-ample class \(A\) with \(c_1(Y)\cdot A>0\) and a normal ample test configuration \((\mathcal Y,\mathcal A)\) for \((Y,A)\) whose total space is smooth, with reduced central fibre, such that
\[
        \Fcal^{\NA}(\mathcal Y,\mathcal A)<0.
\]
Then for every birational morphism \(p:X\to Y\) from a smooth projective surface, \(\Kah_X^{\mathrm{psc}}\subsetneq\mathcal T_X^+\).
\end{proposition}
\begin{proof}
By \cite[Theorem~II.11]{Beauville1996}, every birational morphism between smooth projective surfaces factors as a finite sequence of blowups. Thus write
\[
X=X_N\xrightarrow{p_N}X_{N-1}\longrightarrow\cdots \longrightarrow X_1\xrightarrow{p_1}X_0=Y, \qquad X_i=\operatorname{Bl}_{q_i}X_{i-1},
\]
where the points \(q_1,...,q_N\) may be infinitely near. Denote by \(D_i\subset X_i\) the exceptional divisor of \(p_i\).

We explain how to lift one blowup. Let \(\varpi:\mathcal Y\to\mathbb A^1\) be the test configuration in the assumption, with its given \(\mathbb C^*\)-action lifting multiplication on \(\mathbb A^1\). We identify \((\mathcal Y_1,\mathcal A|_{\mathcal Y_1})\) with \((Y,A)\). The action gives an equivariant trivialization \(\Phi:Y\times\mathbb C^*\xrightarrow{\sim}\varpi^{-1}(\mathbb C^*)\) with \((y,t)\longmapsto t\cdot y\), where \(\lambda\cdot(y,t)=(y,\lambda t)\) on the product, and \(\Phi^*(\mathcal A|_{\varpi^{-1}(\mathbb C^*)})=\operatorname{pr}_Y^*A\). The orbit of \(q_1\in Y\) therefore defines a \(\mathbb C^*\)-equivariant section \(s:\mathbb C^*\rightarrow\mathcal Y\) with \(s(t)=t\cdot q_1\). Since \(\varpi\) is proper, \(s\) extends uniquely across the origin. Denote \(s(\mathbb A^1)\) by \(\mathcal Z_1\). Since \(\varpi \circ s =\operatorname{id}\), \(\varpi\) is a submersion along \(\mathcal Z_1\). Hence \(\mathcal Z_1\) is a smooth codimension-two center meeting each fibre in its smooth locus.

Blowing up \(\mathcal Z_1\), we obtain \(P_1:\widetilde{\mathcal Y}=\operatorname{Bl}_{\mathcal Z_1}\mathcal Y\rightarrow\mathcal Y\) with exceptional divisor \(\mathcal D\). For every sufficiently small rational \(\varepsilon>0\), \(\widetilde{\mathcal A}_\varepsilon=P_1^*\mathcal A-\varepsilon\mathcal D\) is relatively ample and defines a test configuration for \(p_1^*A-\varepsilon D\) on \(\operatorname{Bl}_{q_1}Y\), where \(p_1:\widetilde Y=\operatorname{Bl}_{q_1}Y\to Y\) is the blowdown map and \(D\subset\widetilde Y\) is the exceptional curve. The total space remains smooth, and the central fibre remains reduced.

The non-Archimedean \(\mathcal F\)-functional is given by intersection numbers on the compactified test configuration as in \eqref{eq:JNA-FNA}. For the above fixed blowup, these intersection numbers are polynomial in \(\varepsilon\), while the general-fibre volume tends to \(A^2\). Consequently,
\begin{equation}
\lim_{\varepsilon\searrow0}\Fcal^{\NA}(\widetilde{\mathcal Y}, \widetilde{\mathcal A}_\varepsilon)= \Fcal^{\NA}(\mathcal Y,\mathcal A).
\label{eq:blowup-slope-continuity}
\end{equation}
Since the right-hand side is strictly negative, \(\Fcal^{\NA}(\widetilde{\mathcal Y},\widetilde{\mathcal A}_\varepsilon)\) is negative for every sufficiently small \(\varepsilon>0\).

We now iterate this procedure. At the \(i\)-th stage, after the previous parameters have been fixed, the point \(q_i\in X_{i-1}\) again defines an orbit section of the lifted test configuration. Choose successively \(0<\varepsilon_i\ll1\) so that relative ampleness and strict negativity of the slope are preserved at each stage. We obtain a rational K\"ahler class \(A_N\) on \(X\), defined recursively by
\[
A_i=p_i^*A_{i-1}-\varepsilon_iD_i, \qquad A_0=A,
\]
and a normal ample test configuration \((\mathcal X_N,\mathcal A_N)\) for \((X_N,A_N)\) satisfying
\begin{equation}\label{eq:F-contradiction}
  \Fcal^{\NA}(\mathcal X_N,\mathcal A_N)<0.
\end{equation}
Moreover,
\[
\begin{aligned}
c_1(X_i)\cdot A_i&= \bigl(p_i^*c_1(X_{i-1})-D_i\bigr) \cdot \bigl(p_i^*A_{i-1}-\varepsilon_iD_i\bigr)\\
        &=c_1(X_{i-1})\cdot A_{i-1}-\varepsilon_i.
\end{aligned}
\]
The parameters can therefore be chosen so that \(c_1(X)\cdot A_N>0\). Thus \(2\pi A_N\in\mathcal T_X^+\).

If \(2\pi A_N\) contained a PSC K\"ahler metric, Lemma \ref{lem:PSC-NA-slope} would force \(\Fcal^{\NA}\ge \delta J^{\NA}\ge 0\) for every normal ample test configuration. This contradicts \eqref{eq:F-contradiction}. Hence
\[
        2\pi A_N  \in \mathcal T_X^+\setminus\Kah_X^{\mathrm{psc}},
\]
which proves the proposition.
\end{proof}

We now apply this principle to the test configuration associated with the Harder–Narasimhan filtration.

\begin{proof}[Proof of Theorem~\ref{thm:main-birational-obstruction}]
Since \(X\) is K\"ahler and bimeromorphic to the projective surface \(Y\), it is projective. Choose positive rational numbers \(a,b\) satisfying
\begin{equation}
        \max\left\{g-1,\frac{\Delta(E)}{2}\right\}a
        <b<
        \left(g-1+\frac{\Delta(E)}{2}\right)a.
\label{eq:negative-class-choice}
\end{equation}
Then \(A=a\zeta+bF\) is ample and satisfies \(c_1(Y)\cdot A>0\). Consider the test configuration constructed as in the proof of Proposition \ref{prop:HN-obstruction}. It follows from \eqref{eq:HN-JNA-FNA} that
\[
\Fcal^{\NA}(\mathcal Y,\mathcal A)=\frac{(2-2g-\Delta(E))a^2 + 2ab}{A^2} < 0
\]
Proposition~\ref{prop:birational-propagation} now applies to the given morphism \(q:X\to Y\), which completes the proof.
\end{proof}

\subsection{Proof of Theorem \ref{thm:ruled-descent}} \label{sec:psc-descent}

Fix a class \(\beta\in H^{1,1}(Y,\mathbb R)\) on a ruled surface \(\pi_Y:Y=\Pj(E)\to\Sigma_g\) such that \(\beta=p_*\alpha\), where \(p:X\to Y\) is a composition of \(N\) point blowups and \(\alpha\) is a PSC K\"ahler class on \(X\). Relative to a prescribed section \(C\subset Y\), the blowup points may lie on \(C\), lie in the same ruling fibre, or be infinitely near. Such configurations make the calculations of intersection numbers more involved.

The following lemma replaces \(X\) by a deformation \(X'\) obtained by blowing up \(N\) distinct points of \(Y\setminus C\) in pairwise distinct fibres, which also provides a PSC K\"ahler class \(\alpha'\) on \(X'\) such that \((p')_*\alpha'=\beta\), where \(p':X'\to Y\) is the blowdown morphism. Thus the pushforward class remains fixed, while the geometry of the blowup centres is arranged for the calculation in Proposition~\ref{lem:separated-hn-degeneration}.

\begin{lemma}
\label{lem:separate-blowup-centres}
Let \(\pi_Y:Y=\Pj(E)\to\Sigma_g\) be a ruled surface, let \(C\subset Y\) be a section, and let \(p:X\to Y\) be a composition of \(N\) point blowups. Suppose that \(\omega\) is a PSC K\"ahler metric on \(X\), and write \(\beta=p_*[\omega]\). Then there exist points \(q'_1,\ldots,q'_N\in Y\setminus C\), with \(\pi_Y(q'_1),\ldots,\pi_Y(q'_N)\) pairwise distinct, and a PSC K\"ahler metric \(\omega'\) on \(X'=\operatorname{Bl}_{q'_1,\ldots,q'_N}Y\) such that, for the morphism \(p':X'\to Y\),
\begin{equation}
 (p')_*[\omega']=\beta.
 \label{eq:separation-fixed-pushforward}
\end{equation}
\end{lemma}

\begin{proof}
The case \(N=0\) is immediate. For \(N \ge 1\), fix a factorization
\[
 X=X^{(N)}\xrightarrow{p_N}X^{(N-1)}\longrightarrow\cdots\longrightarrow X^{(1)}\xrightarrow{p_1}X^{(0)}=Y,\qquad p=p_1\circ\cdots\circ p_N,
\]
where \(p_j\) is the blowup of \(q_j\in X^{(j-1)}\). Let \(\mathbb D\subset\mathbb C\) be an open disc centred at \(0\). Let \(\mathcal X^{(0)}=Y\times\mathbb D\) and let \(P^{(0)}=\operatorname{id}_{Y\times\mathbb D}\). We keep \(X^{(j)}\) as the central fibre and deform it step by step.

At the \(j\)-th step, suppose that \(P^{(j-1)}:\mathcal X^{(j-1)}\to Y\times\mathbb D\) has been constructed, with \(\mathcal X^{(j-1)}\to\mathbb D\) smooth and proper, and with central fibre \(X^{(j-1)}\). For \(i<j\), let
\[
 \mathcal U_i=\{(y,t)\in Y\times\mathbb D: \pi_Y(y)=\pi_Y(q'_i(t))\},
\]
and define
\[
 \mathcal V_j:=\operatorname{Exc}(P^{(j-1)}) \cup (P^{(j-1)})^{-1}(C\times\mathbb D) \cup\bigcup_{i<j}(P^{(j-1)})^{-1}(\mathcal U_i).
\]
In particular, we can choose a holomorphic section \(s_j:\mathbb D\to\mathcal X^{(j-1)}\) such that \(s_j(0)=q_j\) and \(s_j(t)\notin\mathcal V_j\) for \(t\ne0\). To see this, take relative coordinates \((z_1,z_2,t)\) on a neighborhood \(U\) of \(q_j=(0,0,0)\). If \(q_j\notin\mathcal V_j\), shrink \(U\) so that \(U\cap\mathcal V_j=\varnothing\). After shrinking \(\mathbb D\), the section \(s_j(t)=(0,0,t)\) then has the required properties. If \(q_j\in\mathcal V_j\), shrink \(U\) and write \(\mathcal V_j\cap U=\{f:=f_1\cdots f_k=0\}\) where \(f_1,\ldots,f_k\) are nonzero local defining functions for the components through \(q_j\). For a general \(v\in\mathbb C^2\), the germ \(f(tv,t)\) is not identically zero, since the lowest order nonzero homogeneous term of \(f\), evaluated at \((v,1)\), is a nonzero polynomial in \(v\). The zeros of \(t\mapsto f(tv,t)\) are isolated. We therefore have \(f(tv,t)\ne0\) for every \(t\ne0\) after shrinking \(\mathbb D\). Thus \(s_j(t)=(tv,t)\) is the required section.

Blowing up along the curve \(\mathcal Z_j:=s_j(\mathbb D)\), we obtain
\[
 P_j:\mathcal X^{(j)} :=\operatorname{Bl}_{\mathcal Z_j}\mathcal X^{(j-1)}\longrightarrow\mathcal X^{(j-1)}, \qquad P^{(j)}=P^{(j-1)}\circ P_j.
\]
As in the proof of Proposition~\ref{prop:birational-propagation}, the resulting family is smooth and proper, and the restriction of \(P_j\) to each fibre is the point blowup at \(s_j(t)\). Writing \(X^{(i)}_t\) for the fibre of \(\mathcal X^{(i)}\to\mathbb D\), we have
\[
 X^{(j)}_t \simeq\operatorname{Bl}_{s_j(t)}X^{(j-1)}_t.
\]
At \(t=0\), we have \(s_j(0)=q_j\), so this is precisely \(p_j:X^{(j)}\to X^{(j-1)}\). Write \(q'_j(t)=\pr_Y(P^{(j-1)}(s_j(t)))\). For \(t\ne0\), the center \(s_j(t)\) lies outside \(\operatorname{Exc}(P^{(j-1)})\). Hence \(\pr_Y\circ P^{(j-1)}\) identifies a neighbourhood of this center in \(X^{(j-1)}_t\) with a neighbourhood of \(q'_j(t)\) in \(Y\). By the choice of \(s_j\), the point \(q'_j(t)\) lies outside \(C\) and in a fibre distinct from those containing \(q'_i(t)\) for every \(i<j\). Thus this step adds an ordinary blowup at a new point of \(Y\).

After \(N\) steps, write \(\varpi:\mathcal X=\mathcal X^{(N)}\to\mathbb D\), \(P=P^{(N)}\), \(X_t=\varpi^{-1}(t)\), and \(p_t=\pr_Y\circ P|_{X_t}\). Then
\[
 X_0=X,\qquad p_0=p,\qquad X_t\simeq\operatorname{Bl}_{q'_1(t),\ldots,q'_N(t)}Y \quad \text{for }t\ne0.
\]
Let \(D_j\) and \(\mathcal D_j\) be the total transforms on \(X\) and \(\mathcal X\) of the exceptional divisors created at the \(j\)-th blowup in the respective sequences, and write \(D_{j,t}=\mathcal D_j|_{X_t}\), so that \(D_{j,0}=D_j\). Since \(\omega\) is a PSC K\"ahler metric on \(X\), we have
\[
 [\omega]=p^*\beta-\sum_{j=1}^N\eps_j[D_j]
\]
for a K\"ahler class \(\beta \in \mathcal{K}_Y\). Choose a smooth closed real \((1,1)\)-forms \(\theta_j\) representing \([\mathcal D_j]\), and a smooth closed real \((1,1)\)-form \(\eta\) representing \(\beta\) on \(Y\). Define
\[
 \Theta=(\pr_Y\circ P)^*\eta -\sum_{j=1}^N\eps_j\theta_j.
\]
Since \([\Theta|_{X_0}]=[\omega]\), Lemma~\ref{lem:psc-deformation} gives, after shrinking \(\mathbb D\), a smooth family of PSC K\"ahler metrics \(\omega_t\), with \(\omega_0=\omega\), satisfying
\[
 [\omega_t]=p_t^*\beta-\sum_{j=1}^N\eps_j[D_{j,t}].
\]
Since each \(D_{j,t}\) is contracted by \(p_t\), we have
\[
 (p_t)_*[\omega_t]=\beta \qquad \text{for }t\in\mathbb D.
\]
Choosing any \(t\ne0\) and taking \(q'_i=q'_i(t)\) and \((X',p',\omega')=(X_t,p_t,\omega_t)\) proves the lemma.
\end{proof}

We now apply the construction and intersection formulas of Section~\ref{sec:basic-intersections} in the above setting when \(E\) is unstable.

\begin{proposition}\label{lem:separated-hn-degeneration}
Let \(\pi_Y:Y=\Pj(E)\to\Sigma_g\), where \(g\ge2\) and \(E\) is unstable, and let \(C\) be the Harder--Narasimhan section. Suppose that \(p:X\to Y\) is the blowup of points \(q_1,\ldots,q_N\) outside \(C\) in pairwise distinct fibres, with exceptional divisors \(D_1,\ldots,D_N\). Let \(A=p^*B-\sum_{i=1}^N\eps_iD_i\) be a \(\mathbb Q\)-ample class on \(X\), where \(B=a\zeta+bF\) is a \(\mathbb Q\)-ample class on \(Y\). If
\begin{equation}
 \sum_{i=1}^N\eps_i<b+\frac{\Delta(E)}2a,
 \label{eq:separated-assumption}
\end{equation}
then \((X,A)\) admits an ample test configuration \((\mathcal{X},\mathcal{A})\) with smooth total space and smooth central fibre such that \(J^{\NA}(\mathcal{X},\mathcal{A})>0\) and
\begin{equation}
 \Fcal^{\NA}(\mathcal{X},\mathcal{A})=\frac{2a\left[b-\left(g-1+\frac{\Delta(E)}2\right)a\right]}{A^2}.
 \label{eq:separated-hn-slope}
\end{equation}
\end{proposition}

\begin{proof}
Consider the compactified family \(\overline P:\overline{\mathcal X}=\operatorname{Bl}_{\coprod_i\mathcal Z_i}\overline{\mathcal Y}\rightarrow\overline{\mathcal Y}\) constructed in Section~\ref{sec:basic-intersections}, with exceptional divisors \(\mathcal D_i\). Define
\[
 \overline{\mathcal B}=a\bar\zeta+b\bar F-\frac a2\mathcal Y_0,\qquad\overline{\mathcal A} =\overline P^{\,*}\overline{\mathcal B} -\sum_{i=1}^N\eps_i\mathcal D_i,\qquad \mathcal A=\overline{\mathcal A}|_{\mathcal X}.
\]
The induced family is smooth, with general fibre \((X,A)\) and
\[
 \mathcal X_0=\operatorname{Bl}_{s_1(0),\ldots,s_N(0)}\mathcal Y_0,\qquad \mathcal Y_0=\Pj(L\oplus Q),\qquad s_i(0)\in\Pj(L).
\]

Except on the central fibre, the class \(\mathcal A\) corresponds to the ample class \(A\) on \(X\). It therefore remains to prove that \(\mathcal A_0:=\mathcal A|_{\mathcal X_0}\) is ample. Let \(\widetilde C\) and \(\widetilde{\Pj(L)}\) denote the strict transforms of \(C=\Pj(Q)\) and \(\Pj(L)\), respectively. Write \(D_{i,0}=\mathcal D_i|_{\mathcal X_0}\), let \(\widetilde F_i\) be the strict transform of the fibre through \(s_i(0)\), and let \(F\) denote the class of an unaffected fibre. Then
\[
\begin{gathered}
 \mathcal A_0\cdot\widetilde C=b-\frac{\Delta(E)}2a,\qquad \mathcal A_0\cdot\widetilde{\Pj(L)} =b+\frac{\Delta(E)}2a-\sum_{i=1}^N\eps_i,\\
 \mathcal A_0\cdot F=a,\qquad\mathcal A_0\cdot D_{i,0}=\eps_i,\qquad\mathcal A_0\cdot\widetilde F_i=a-\eps_i,\qquad 1\le i\le N.
\end{gathered}
\]
By \eqref{eq:separated-assumption}, \(\mathcal A_0\cdot\widetilde{\Pj(L)}>0\). Therefore, all intersection numbers listed above are positive thanks to the ampleness of \(A\). These curves generate the cone of effective curve classes on \(\mathcal X_0\). Since \(\mathcal A_0^2=A^2>0\), the Nakai--Moishezon criterion shows that \(\mathcal A_0\) is ample. Thus \(\mathcal A\) is relatively ample. Consequently, \((\mathcal X,\mathcal A)\) is a normal ample test configuration for \((X,A)\).

We next verify \eqref{eq:NA-normalization}. Let \(\mathcal Z\subset\overline{\mathcal Y}\) be the orbit section of a general point of \(Y\). As in the proof of Proposition~\ref{prop:HN-obstruction}, we have \(\overline{\mathcal B}\cdot\mathcal Z=0\). We may choose this point outside \(C\) and outside the ruling fibres containing \(q_i\). Then \(\mathcal Z\) is disjoint from every \(\mathcal Z_i\), and \(\overline{\mathcal A}\cdot \mathcal Z=0\) after identifying \(\mathcal Z\) with its strict transform on \(\overline{\mathcal X}\). The same argument as in the proof of Proposition~\ref{prop:HN-obstruction}, on the common resolution of Section~\ref{sec:NA}, therefore gives \(\overline{\mathcal A}\cdot H^2=0\).

Since
\[
 K_{\overline{\mathcal Y}/\Pj^1}=-2\bar\zeta+(2g-2)\bar F,\qquad K_{\overline{\mathcal X}/\Pj^1} =\overline P^{\,*}K_{\overline{\mathcal Y}/\Pj^1} +\sum_{i=1}^N\mathcal D_i,
\]
by \eqref{eq:basic-extension-triple-intersections}, \eqref{eq:basic-curve-blowup}, and \eqref{eq:basic-orbit-exceptional-intersections}, we obtain
\[
\begin{aligned}
 K_{\overline{\mathcal X}/\Pj^1}\cdot\overline{\mathcal A}^{\,2} &=K_{\overline{\mathcal Y}/\Pj^1}\cdot\overline{\mathcal B}^{\,2}  +\sum_{i=1}^N\eps_i^2(1-1) =2a\left[b-\left(g-1+\frac{\Delta(E)}2\right)a\right],\\
 \overline{\mathcal A}^{\,3} &=\overline{\mathcal B}^{\,3}+\sum_{i=1}^N\eps_i^3 =\frac{\Delta(E)}2a^3-3a^2b+\sum_{i=1}^N\eps_i^3.
\end{aligned}
\]
Since the central fibre is smooth and reduced, by \eqref{eq:JNA-FNA}, together with \eqref{eq:separated-assumption}, we obtain \eqref{eq:separated-hn-slope} and \(J^{\NA}(\mathcal X,\mathcal A)>0\).
\end{proof}

\vspace{2mm}

\begin{proof}[Proof of Theorem~\ref{thm:ruled-descent}]
Let \(2\pi A\in\Kah_X^{\mathrm{psc}}\), and write \(B=p_*A\) on \(Y\). Since \(Y\) is projective, \(X\) is also projective. Fix a factorization of \(p\) into \(N\) point blowups, and let \(D_j\) denote the total transform on \(X\) of the exceptional divisor produced at the \(j\)-th step. Then \(A=p^*B-\sum_{j=1}^N\eps_jD_j\) for some \(\eps_i>0\).

We first show that \(B\) is K\"ahler. Let \(R \subset Y\) be any effecitve curve and let \(\widetilde R \subset X\) be its strict transform. Since \(A\) is K\"ahler,
\[
0< A \cdot \widetilde R = B \cdot R - \sum_{j=1}^N\eps_jD_j\cdot  \widetilde R.
\]
Therefore \(B \cdot R >0\). Combing with \(B^2=A^2 + \sum_{j=1}^N\eps_j^2>0\), we conclude \(B\) is K\"ahler thanks to the Nakai--Moishezon criterion. Since \(c_1(X)=p^*c_1(Y)-\sum_jD_j\), we also have
\begin{equation} \label{eq:descent-positive-total}
 0<c_1(X)\cdot A =c_1(Y)\cdot B-\sum_{j=1}^N\eps_j.
\end{equation}
In particular, \(c_1(Y)\cdot B>0\) and \(2\pi B \in \mathcal{T}_Y^+\).

If \(Y\) is rational, then \(Y\) is toric, including the case \(Y=\mathbb F_1\). Thus \(2\pi B\in\Kah_Y^{\mathrm{psc}}\) by \cite[Theorem~C]{Sha2026Variational}. If \(Y\) is non-rational, with \(g=1\) or \(E\) slope-semistable, we also have \(2\pi B\in\Kah_Y^{\mathrm{psc}}\) thanks to Theorem~\ref{thm:main-minimal}.

It remains to consider \(Y=\Pj(E)\xrightarrow{\pi}\Sigma_g\) with \(g\ge2\) and \(E\) slope-unstable. Write \(B=a\zeta+bF\). First assume that \(A\) is rational, and suppose for a contradiction that \(2\pi B\notin\Kah_Y^{\mathrm{psc}}\). By Theorem~\ref{thm:main-minimal}, we have
\begin{equation}
 b\le\left(g-1+\frac{\Delta(E)}2\right)a.
 \label{eq:descent-wall}
\end{equation}
Apply Lemma~\ref{lem:separate-blowup-centres} with \(C\) the Harder--Narasimhan section. We obtain a blowup \(p':X'\to Y\) at points outside \(C\) in pairwise distinct fibres, together with a class \(A'=(p')^*B-\sum_{j=1}^N\eps_jD'_j\in\Kah_{X'}^{\mathrm{psc}}\), where \(D'_j\) are exceptional divisors. In particular, the construction preserves \(B\) and the coefficients \(\eps_j\), so \(A'\) is rational. Since \(c_1(Y)\cdot B=2b+(2-2g)a\), by \eqref{eq:descent-positive-total} and \eqref{eq:descent-wall}, we obtain
\[
 \sum_{j=1}^N\eps_j<2b-2(g-1)a\le b+\frac{\Delta(E)}2a-(g-1)a <b+\frac{\Delta(E)}2a.
\]
Thus \eqref{eq:separated-assumption} holds for \((X',A')\). By Proposition~\ref{lem:separated-hn-degeneration}, there is a normal ample test configuration for \((X',A')\) with \(\Fcal^{\NA}\le0\) and \(J^{\NA}>0\). This contradicts Lemma~\ref{lem:PSC-NA-slope}, since \(2\pi A'\in \cpsc\). Hence \(2\pi B\in\Kah_Y^{\mathrm{psc}}\) whenever \(A\) is rational.

Now let \(A\) be real. By Lemma~\ref{lem:psc-openness}, choose \(s>0\) sufficiently small that \(2\pi(A-s p^*F)\in\Kah_X^{\mathrm{psc}}\). The classes \(p^*\zeta,p^*F,D_1,\ldots,D_N\) form a basis of \(H^{1,1}(X,\mathbb R)\). We can therefore choose rational classes \(A_k\to A-s p^*F\) with \(2\pi A_k\in\Kah_X^{\mathrm{psc}}\). Write \(p_*A_k=a_k\zeta+b_kF\). Since, \(A_k\in\Kah_X^{\mathrm{psc}}\), we have
\[
 b_k>\left(g-1+\frac{\Delta(E)}2\right)a_k.
\]
Letting \(k \rightarrow \infty\), it follows that
\[
 b-s\ge\left(g-1+\frac{\Delta(E)}2\right)a, \qquad\text{hence}\qquad b>\left(g-1+\frac{\Delta(E)}2\right)a.
\]
Theorem~\ref{thm:main-minimal} now yields \(2\pi B\in\Kah_Y^{\mathrm{psc}}\). This proves \(p_*\Kah_X^{\mathrm{psc}}\subseteq\Kah_Y^{\mathrm{psc}}\).

\vspace{2mm}

For the reverse inclusion, let \(\alpha\in\Kah_Y^{\mathrm{psc}}\). Applying \cite[Theorem~6.8]{Brown2026} successively along the fixed factorization of \(p\), with the parameters chosen sufficiently small at each step, gives a class
\[
 \alpha_X=p^*\alpha-\sum_{j=1}^N\eps'_jD_j \in\Kah_X^{\mathrm{psc}}.
\]
Since \(p_*D_j=0\), we have \(p_*\alpha_X=\alpha\). This proves \(p_*\Kah_X^{\mathrm{psc}}=\Kah_Y^{\mathrm{psc}}\).
\end{proof}

Theorem~\ref{thm:ruled-descent} therefore determines the image of the whole PSC K\"ahler cone under \(p_*\), without a smallness assumption on the blowup parameters of the given class on \(X\). Together with Theorem~\ref{thm:main-minimal}, it also gives another proof of Theorem~\ref{thm:main-birational-obstruction}: choose a class in \(\mathcal T_Y^+\setminus\Kah_Y^{\mathrm{psc}}\) and pullback it to a K\"ahler class on \(X\) with sufficiently small blowup parameters. The pullback class still has positive total scalar curvature, while it can not admit any PSC K\"ahler metric.

\subsection{Existence of PSC K\"ahler metrics as a birational invariant on surfaces} \label{sec:birational-existence}

We now prove Corollary~\ref{cor:psc-birational-invariance} by pushing forward the given PSC K\"ahler class to the chosen minimal model and then applying \cite[Theorem~6.8]{Brown2026} to pull back this prescribed class to the birational model.

\begin{proof}[Proof of Corollary~\ref{cor:psc-birational-invariance}]
Fix \(\alpha\in\Kah_X^{\mathrm{psc}}\). By \cite[Theorem~B]{Brown2026}, \(\kappa(Y_{\mathrm{min}})=\kappa(X)=-\infty\). By Enriques–Kodaira classification, \(Y_{\mathrm{min}}\) is \(\Pj^2\) or a geometrically ruled surface \cite[Chapter~VI]{BHPV2004}. By Theorem~\ref{thm:ruled-descent}, we have
\[
 \gamma:=(r\circ p)_*\alpha \in\Kah_{Y_{\mathrm{min}}}^{\mathrm{psc}}.
\]
Applying \cite[Theorem~6.8]{Brown2026} successivly along a factorization of \(r\) then gives a class \(\beta\in\Kah_Y^{\mathrm{psc}}\) with \(r_*\beta=\gamma=(r\circ p)_*\alpha\).
\end{proof}

\appendix 
\section{The momentum construction on decomposable ruled surfaces}
\label{app:decomposable-model}

In the appendix, we give a specific construction of the PSC K\"ahler cone on decomposable minimal ruled surfaces by a method different from that used in the main text.

\subsection{The \texorpdfstring{\(S^1\)}{S1}-invariant PSC K\"ahler cone}
\label{subsec:S1-invariant-PSC-cone}

Let \(X=\mathbb P(\mathcal O_{\Sigma_g}\oplus L)\xrightarrow{\ \pi\ } \Sigma_g\) with \(g\ge1\) and \(\deg L=-d<0\). The splitting \(\mathcal O_{\Sigma_g}\oplus L\) induces a natural fibrewise \(S^1\)-action on \(X\), obtained by scalar multiplication on the \(L\)-factor.  We define the \(S^1\)-invariant positive scalar curvature cone by
\[
\mathcal K_{X,S^1}^{\mathrm{psc}}:=\left\{\alpha\in\mathcal K_X:\,\exists\, \omega \in \alpha,~S^1\text{-invariant, s.t. }S(\omega)>0\right\}.
\]
The purpose of this subsection is to compute this cone. In particular, any K\"ahler class on \(X\) can be written as \(\alpha_{\lambda,m}:=2\pi\bigl(mC_0+d(\lambda+m)F\bigr)\). 

Choose a constant-scalar-curvature K\"ahler metric \(\omega_\Sigma\) on \(\Sigma_g\) satisfying \([\omega_\Sigma]=-2\pi c_1(L)\). Then
\[
\int_{\Sigma_g}\omega_\Sigma=2\pi d, \qquad \text{and} \qquad S(\omega_\Sigma) = \frac{2-2g}{d}:=s\le0,
\]
thanks to the Gauss--Bonnet formula. We use the admissible momentum construction on decomposable ruled surfaces, in the spirit of Calabi ansatz \cite{Calabi1982}, following the Hwang--Singer momentum construction for circle-invariant K\"ahler metrics \cite{HwangSinger2002}.

Let \(L^\times\) denote the complement of the zero section in the total space of \(L\).  Choose a Hermitian metric \(h\) on \(L\) such that, on \(L^\times\),
\[
 \ddbar\log |v|_h^2  =  \pi^*\omega_\Sigma.
\]
Define \(\tau:=\log |v|_h^2\), and define the real one-form
\[
\eta:=\frac{\sqrt{-1}}{2} (\bar\partial-\partial)\tau .
\]
Then
\[
 d\eta=\ddbar\tau=\pi^*\omega_\Sigma .
\]
On each fibre, if \(v=\rho e^{\sqrt{-1}\theta}\) in a local trivialization, then \(\eta=d\theta\).  Thus \(\eta\) is the connection one-form induced by the Hermitian metric \(h\) on the principal \(S^1\)-bundle of unit vectors in \(L\), extended radially to \(L^\times\).

We now choose an admissible radial potential.  Let \(f\) be a smooth strictly convex function such that \(f'(\mathbb R)=(0,m)\). Define the momentum coordinate \(r=f'(\tau)\). Since \(f'\) is a diffeomorphism from \(\mathbb R\) onto \((0,m)\), there is a smooth function \(\phi\) on \((0,m)\) determined by
\[
\phi\left(f'(\tau)\right)=f''(\tau).
\]
We require that \(\phi\) extends smoothly to \([0,m]\), such that \(\phi(r)>0\) for \(0<r <m\), and satisfies the standard endpoint conditions
\begin{equation}\label{eq:phi-endpoint}
\phi(0)=\phi(m)=0,  \qquad  \phi'(0)=1,  \qquad  \phi'(m)=-1.        
\end{equation}

Conversely, every smooth function \(\phi\) satisfying these positivity and endpoint conditions arises from such a function \(f\): indeed, one solves
\[
\frac{dt}{dr}=\frac{1}{\phi(r)}
\]
on \((0,m)\).  The endpoint conditions \eqref{eq:phi-endpoint} imply \(t\to-\infty\quad\text{as }r\to0\) and \(t\to+\infty\quad\text{as }r\to m\), so that \(t:(0,m)\to\mathbb R\) is a diffeomorphism.  Its inverse \(r=r(t)\) then defines \(f\), up to an additive constant, by \(f'(t)=r(t)\). Then
\[
f''(t)=\phi(r(t))>0.
\]

On \(L^\times\), define the admissible metric
\[
 \omega:=\lambda\,\pi^*\omega_\Sigma+\ddbar f(\tau).
\]
Since \(dr=f''(\tau)d\tau=\phi(r)d\tau\), we have
\[
\begin{aligned}
 \ddbar f(\tau) &= f'(\tau)\ddbar\tau + f''(\tau)\sqrt{-1}\,\partial\tau\wedge\bar\partial\tau \\
&= r\,\pi^*\omega_\Sigma + dr\wedge\eta .
\end{aligned}
\]
Therefore \(\omega=(\lambda+r)\pi^*\omega_\Sigma+dr\wedge\eta\), and it is K\"ahler on \(L^\times\) precisely when \(\lambda+r>0\) and \(\phi(r)>0\). Since \(0<r<m\), the first condition follows from \(\lambda>0\). The endpoint conditions \eqref{eq:phi-endpoint} guarantee that \(\omega\) extends smoothly to \(X\), and its K\"ahler class is
\[
[\omega] =2\pi\bigl(mC_0+d(\lambda+m)F\bigr)= \alpha_{\lambda,m}.
\]
Define \(G(r):=(\lambda+r)\phi(r)\). Then the endpoint conditions \eqref{eq:phi-endpoint} become
\begin{equation}\label{eq:G-endpoint}
 G(0)=G(m)=0,\qquad G'(0)=\lambda, \qquad  G'(m)=-(\lambda+m).
\end{equation}
We next compute the scalar curvature.  Since
\begin{equation}\label{eq:app-vol}
\omega^2  =2(\lambda+r)\,\pi^*\omega_\Sigma\wedge dr\wedge\eta,
\end{equation}
the Ricci form is
\begin{equation} \label{eq:ricci-admissible}
\Ric(\omega)=\pi^*\Ric(\omega_\Sigma)-\ddbar\log G .
\end{equation}
For every smooth function \(u=u(r)\), one has
\begin{equation} \label{eq:ddbar-u}
\ddbar u=(\phi u')'\,dr\wedge\eta+\phi u'\,\pi^*\omega_\Sigma,
\end{equation}
where primes denote derivatives with respect to \(r\).  Applying this to \(u=\log G\), we get
\[
\phi(\log G)'= \frac{G'}{\lambda+r}.
\]
Hence
\[
(\lambda+r)\bigl(\phi(\log G)'\bigr)' + \phi(\log G)' =G''.
\]
Since \(\Ric(\omega_\Sigma)=s\,\omega_\Sigma\), equations \eqref{eq:ricci-admissible} and \eqref{eq:ddbar-u} give
\[
 2\Ric(\omega)\wedge\omega=2\bigl(s-G''(r)\bigr)\pi^*\omega_\Sigma\wedge dr\wedge\eta.
\]
Therefore, we obtain
\begin{equation}
\label{eq:admissible-scalar-curvature}
S(\omega) = \frac{2\Ric(\omega)\wedge\omega}{\omega^2}= \frac{s-G''(r)}{\lambda+r}.
\end{equation}

\begin{lemma}
\label{lem:admissible-positive-scalar-curvature}
The class \(\alpha_{\lambda,m}:=2\pi\bigl(mC_0+d(\lambda+m)F\bigr)\) contains an admissible K\"ahler metric with positive scalar curvature if and only if \(2\lambda+sm>0\).
\end{lemma}

\begin{proof}
Suppose that an admissible positive scalar-curvature metric exists. Then  \(G\) satisfies the endpoint condition \eqref{eq:G-endpoint} and, by the scalar-curvature formula,
\[
G''(r)<s\qquad \text{for }0<r<m.
\]
Equivalently, we may write \(G''=s-\mu\) for a smooth positive function \(\mu\) on \((0,m)\). Indeed, Taylor's expansion with integral remainder gives
\[
G(m)=G(0)+mG'(0)+\int_0^m(m-r)G''(r)\,dr.
\]
Since \(G(0)=G(m)=0\) and \(G'(0)=\lambda\), this becomes
\[
\int_0^m(m-r)G''(r)\,dr=-m\lambda.
\]
Substituting \(G''=s-\mu\), we obtain
\[
\int_0^m(m-r)\mu(r)\,dr=\lambda m+\frac{s}{2}m^2.
\]
The left-hand side is strictly positive, and hence \(2\lambda+sm>0\).

Conversely, assume \(2\lambda+sm>0\).  Define
\[
A:=2\lambda+(s+1)m,\qquad B:=m\left(\lambda+\frac{s}{2}m\right).
\]
Then \(B>0\).  Moreover,
\[
 A=(2\lambda+sm)+m>0,
\]
and
\[
\begin{aligned}
mA-B = m\left(2\lambda+(s+1)m\right) - m\left(\lambda+\frac{s}{2}m\right) =B+m^2>0.
\end{aligned}
\]
Thus \(B<mA\).

Choose a smooth positive function \(\mu\) on \([0,m]\) such that
\begin{equation} \label{eq:mu-moment-conditions}
\int_0^m \mu(r)\,dr=A,\qquad \int_0^m (m-r)\mu(r)\,dr=B.
\end{equation}
This is possible since \(B/A\in(0,m)\). In particular, one can choose a smooth positive probability density on \([0,m]\) whose barycenter with respect to the function \(m-r\) is \(B/A\), and multiply it by \(A\).

Define \(G\) by
\[
G''=s-\mu,\qquad G(0)=0, \qquad  G'(0)=\lambda.
\]
The first condition in \eqref{eq:mu-moment-conditions} gives
\[
 G'(m) = G'(0)+\int_0^m G''(r)\,dr =\lambda+sm-A=-(\lambda+m).
\]
The second condition in \eqref{eq:mu-moment-conditions} gives
\[
\begin{aligned}
G(m)&=G(0)+mG'(0)+\int_0^m (m-r)G''(r)\,dr   \\
&= \lambda m+\frac{s}{2}m^2-B  \\
 &= 0.
\end{aligned}
\]
Therefore \(G\) satisfies the endpoint condition \eqref{eq:G-endpoint}.

Since \(G''=s-\mu<0\), the function \(G\) is strictly concave.  Together with \(G(0)=G(m)=0\), this implies \(G(r)>0\) for \(0<r<m\). Define
\[
\phi(r):=\frac{G(r)}{\lambda+r}.
\]
Then \(\phi>0\) on \((0,m)\), \(\phi\in C^\infty([0,m])\), and satisfies the endpoint condition \eqref{eq:phi-endpoint}.

Thus \(\phi\) defines an admissible K\"ahler metric in \(\alpha_{\lambda,m}\) with positive scalar curvature by \eqref{eq:admissible-scalar-curvature}, because \(G^{\prime\prime}<s\).
\end{proof}

\begin{proposition}
\label{prop:S1-invariant-psc-cone}
The \(S^1\)-invariant positive scalar curvature cone is
\[
\mathcal K_{X,S^1}^{\mathrm{psc}}=\left\{ \alpha_{\lambda,m}\in\mathcal{K}_X: ~2\lambda+sm>0\right\}.
\]
\end{proposition}

\begin{proof}
The inclusion
\(\left\{ \alpha_{\lambda,m}\in \mathcal K_X;~ 2\lambda+sm>0\right\} \subseteq\mathcal K_{X,S^1}^{\mathrm{psc}}\) follows directly from Lemma~\ref{lem:admissible-positive-scalar-curvature}, since admissible metrics are \(S^1\)-invariant.

It remains to prove the reverse inclusion.  Let \(\omega\in\alpha_{\lambda,m}\) be an \(S^1\)-invariant K\"ahler metric with \(S(\omega)>0\). Let \(\xi\) be the real holomorphic vector field generating the fibrewise \(S^1\)-action. Since the holomorphic \(S^1\)-action preserves \(\omega\) and has a nonempty fixed locus, Frankel’s theorem \cite{Frankel1959} implies that the action is Hamiltonian. Thus there exists a moment map \(\mu_\omega\) with \(\iota_\xi\omega=-d\mu_\omega\). The fixed locus of the \(S^1\)-action on \(X\) has two connected components, namely the zero section \(D_0\) and the infinity section \(D_\infty\). Normalize the moment map by \(\mu_\omega|_{D_0}=0\). Since the generator has period \(2\pi\) and
\[
 \int_F\omega=2\pi m,
\]
the difference of the moment map between the two fixed sections is \(m\). Hence \(\mu_\omega|_{D_\infty}=m\). With this normalization, the image of \(\mu_\omega\) is the interval \([0,m]\).  Define
\[
 h_\omega:=m-\mu_\omega.
\]
Then \(h_\omega\ge0\) and \(h_\omega\not\equiv0\). Since \(S(\omega)>0\), it follows that
\[
 I(\omega):= \int_X h_\omega\,S(\omega)\,\omega^2>0.
\]
We claim that \(I(\omega)\) depends only on the K\"ahler class \(\alpha_{\lambda,m}\), with the above normalization of the moment map. Indeed,
\[
\begin{aligned}
I(\omega)=\int_X h_\omega \bigl(S(\omega)-\overline S_{\alpha_{\lambda,m}}\bigr) \,\omega^2 + \overline S_{\alpha_{\lambda,m}} \int_X h_\omega\,\omega^2.
\end{aligned}
\]
Since \(dh_\omega=-d\mu_\omega=\iota_\xi\omega\), the function \(h_\omega\) is a Hamiltonian function for \(-\xi\), under our sign convention. Up to the normalization, the first term is the Futaki invariant of \(-\xi\), which depends only on the K\"ahler class and the holomorphic vector field \cite{FutakiMabuchi1995}. We next show that
\[
I_0=\int_X h_\omega\,\omega^2
\]
also depends only on the K\"ahler class.  Let \(\omega_0,\omega_1\in\alpha_{\lambda,m}\) be two \(S^1\)-invariant K\"ahler forms and let
\[
\omega_t:=(1-t)\,\omega_0+t\,\omega_1, \qquad 0\le t\le1.
\]
Then \(\omega_t\) is an \(S^1\)-invariant K\"ahler form in the same class.  Let \(\mu_t\) be its moment map. Because \(\omega_1-\omega_0\) is exact and \(S^1\)-invariant, there is an \(S^1\)-invariant real one-form  \(\gamma\) such that
\[
\dot\omega_t = \omega_1-\omega_0=d\gamma.
\]
Differentiating \(\iota_\xi\omega_t=-d\mu_t\) with respect to \(t\), we obtain
\[
\iota_\xi d\gamma=-d\dot\mu_t.
\]
Since \(\gamma\) is \(S^1\)-invariant, Cartan's formula gives
\[
\iota_\xi d\gamma =\mathcal L_\xi\gamma-d(\iota_\xi\gamma)=-d(\iota_\xi\gamma).
\]
It follows that
\[
d\bigl(\dot\mu_t-\iota_\xi\gamma\bigr)=0.
\]
Hence \(\dot\mu_t-\iota_\xi\gamma\) is constant. On \(D_0\), the normalization gives \(\dot\mu_t=0\), while \(\xi=0\) because \(D_0\) is fixed. Therefore, the constant vanishes and \(\dot\mu_t=\iota_\xi\gamma\). We now differentiate the first moment of the moment map,
\[
\begin{aligned}
\frac{d}{dt}\int_X\mu_t\omega_t^2&=\int_X\dot\mu_t\omega_t^2+2\int_X\mu_t\,d\gamma\wedge\omega_t            \\
&=\int_X(\iota_\xi\gamma)\omega_t^2-2\int_Xd\mu_t\wedge\gamma\wedge\omega_t         \\
&=\int_X(\iota_\xi\gamma)\omega_t^2+ 2\int_X\iota_\xi\omega_t\wedge\gamma\wedge\omega_t\\
&=0
\end{aligned}
\]
Thus, \(I_0\) is independent of the chosen \(S^1\)-invariant representative \(\omega\in\alpha_{\lambda,m}\).  Since
\[
I_0=\int_Xh_\omega\,\omega^2 = m\int_X\omega^2 - \int_X\mu_\omega\,\omega^2,
\]
and \(\int_X\omega^2\) depends only on the K\"ahler class, the second term is also class-invariant.  Consequently \(I(\omega)\) depends only on \(\alpha_{\lambda,m}\), and we may denote it by \(I(\alpha_{\lambda,m})\). We can compute \(I(\alpha_{\lambda,m})\) by using any admissible metric in the same class.  For an admissible metric \(\omega\), the normalized moment map \(\mu_\omega\) is precisely the momentum coordinate \(r\). Hence \(h_\omega=m-r\). Using \eqref{eq:admissible-scalar-curvature} and \eqref{eq:app-vol}, we obtain, for some positive constant \(C\),
\[
I(\alpha_{\lambda,m}) = C\int_0^m(m-r)\bigl(s-G''(r)\bigr)\,dr.
\]
Using the endpoint conditions \eqref{eq:G-endpoint}, we have
\[
\begin{aligned}
\int_0^m (m-r)G''(r)\,dr &= \bigl[(m-r)G'(r)\bigr]_0^m+ \int_0^m G'(r)\,dr    \\
&=-mG'(0)+G(m)-G(0) \\
 &=-m\lambda.
\end{aligned}
\]
Therefore
\[
\begin{aligned}
I(\alpha_{\lambda,m})&=C\left(\frac{s}{2}m^2+\lambda m\right) =\frac{Cm}{2}(2\lambda+sm)>0.
\end{aligned}
\]
Since \(C>0\) and \(m>0\), we obtain \(2\lambda+sm>0\). This proves the reverse inclusion.
\end{proof}

\subsection{Symmetrization and the full PSC K\"ahler cone}
\label{subsec:symmetrization-full-cone}

We first show that any Lie group action-invariant PSC class must admit a PSC K\"haler metric.

\begin{lemma}
\label{lem:invariant-representative-from-prescribed-measures}
Let a compact Lie group \(G\) act holomorphically on a compact K\"ahler manifold \(Y\), and let \(\beta\in\mathcal K_Y\) be a \(G\)-invariant K\"ahler class.  Assume that \(\beta\in\mathcal K_Y^{\mathrm{psc}}\). Then \(\beta\) contains a \(G\)-invariant K\"ahler metric of positive scalar curvature.
\end{lemma}

\begin{proof}
Choose a representative K\"ahler metric \(\omega_0\in\beta\). Since \(\beta\) is \(G\)-invariant, the averaged form
\[
   \omega:=\int_G g^*\omega_0\,dg
\]
is a \(G\)-invariant K\"ahler metric in the same class \(\beta\), where \(dg\) denotes the normalized Haar measure on \(G\).

Since \(\beta\in\mathcal K_Y^{\mathrm{psc}}\), there exists a K\"ahler metric of positive scalar curvature in \(\beta\). Choose \(\Omega=\overline{S}_\beta\,\omega^n\). By \cite[Theorem~A]{Sha2026Variational}, the prescribed scalar-curvature measure equation
\begin{equation} \label{eq:psc-measure-symmetrization}
S(\omega_\varphi)\,\omega_\varphi^n = \Omega, \qquad\omega_\varphi=\omega+\ddbar\varphi\in\beta,
\end{equation}
admits a smooth solution \(\varphi\).  Moreover, for fixed \(\Omega\), the solution is unique up to addition of constants to \(\varphi\).

For every \(g\in G\), since \(g^*\omega=\omega\) and \(g^*\Omega=\Omega\), the potential \(g^*\varphi\) gives another solution of \eqref{eq:psc-measure-symmetrization}. By uniqueness,  \(g^*\omega_\varphi=\omega_\varphi\). Thus \(\omega_\varphi\) is \(G\)-invariant. Finally,
\[
S(\omega_\varphi) =\frac{\Omega}{\omega_\varphi^n}>0.
\]
Hence \(\omega_\varphi\) is a \(G\)-invariant K\"ahler metric of positive scalar curvature in \(\beta\).
\end{proof}

We now give the concrete formula for the full PSC K\"ahler cone.

\begin{proposition}
\label{prop:explicit-split-full-cone}
Let \(X=\mathbb P(\mathcal O_{\Sigma_g}\oplus L)\rightarrow\Sigma_g\) with \(\deg L=-d<0\), and let  \(\alpha_{\lambda,m}:=2\pi\bigl(mC_0+d(\lambda+m)F\bigr)\) be a K\"ahler class.
Then
\[
\mathcal K_X^{\mathrm{psc}} =\left\{\alpha_{\lambda,m}\in\mathcal{K}_X:~2\lambda-\frac{2m(g-1)}{d}>0\right\}.
\]
In particular, if \(g\ge2\), then \(\mathcal K_X^{\mathrm{psc}}\subsetneq\mathcal T_X^+\).
\end{proposition}

\begin{proof}
By Proposition~\ref{prop:S1-invariant-psc-cone},
\[
 \mathcal K_{X,S^1}^{\mathrm{psc}}=\left\{ \alpha_{\lambda,m} \in \mathcal K_X: ~ 2\lambda+sm>0\right\} \subseteq \mathcal K_X^{\mathrm{psc}}.
\]

Conversely, suppose that \(\alpha_{\lambda,m}\in\mathcal K_X^{\mathrm{psc}}\). The natural fibrewise \(S^1\)-action on \(X\) is connected, hence acts trivially on cohomology.  Therefore the class \(\alpha_{\lambda,m}\) is \(S^1\)-invariant.  Applying Lemma~\ref{lem:invariant-representative-from-prescribed-measures} to this \(S^1\)-action, we obtain an \(S^1\)-invariant K\"ahler metric of positive scalar curvature in the class \(\alpha_{\lambda,m}\).  Hence \(\alpha_{\lambda,m}\in \mathcal K_{X,S^1}^{\mathrm{psc}}\). Thus
\[
\mathcal K_X^{\mathrm{psc}} =\mathcal K_{X,S^1}^{\mathrm{psc}}=\left\{\alpha_{\lambda,m}\in\mathcal{K}_X:~\lambda>0,\ m>0,\ 2\lambda+sm>0\right\}.
\]

It remains to compute \(\mathcal T_X^+\). Since \(c_1(X) = 2C_0+(d+2-2g)F\), it follows from \eqref{eq:decomposable-section} that
\[
\begin{aligned}
c_1(X)\cdot\alpha_{\lambda,m} &=2\pi \bigl(2C_0+(d+2-2g)F\bigr)\cdot \bigl(mC_0+d(\lambda+m)F\bigr)                 \\
        &=2\pi\bigl(-2dm+2d(\lambda+m)+m(d+2-2g)\bigr)                                         \\
        &=2\pi d\left(2\lambda+\left(1+\frac{2-2g}{d}\right)m\right)                                        \\
        &=2\pi d\bigl(2\lambda+(1+s)m\bigr).
\end{aligned}
\]
Since \(d>0\), this gives
\[
\mathcal T_X^+=\left\{\alpha_{\lambda,m}\in\mathcal{K}_X:~ 2\lambda+(1+s)m>0\right\}.
\]
Therefore \(\mathcal K_X^{\mathrm{psc}}\subseteq\mathcal T_X^+\). In particular, the inclusion is strict when \(g\ge 2\). Assume now that \(g\ge2\), so that \(s<0\).  Fix \(m>0\) and take \(\lambda=-sm/2>0\). Then \(2\lambda+sm=0\), whereas \(2\lambda+(1+s)m=m>0\). Hence \(\alpha_{\lambda,m}\in\mathcal T_X^+\setminus\mathcal K_X^{\mathrm{psc}}\). Consequently, \( \mathcal K_X^{\mathrm{psc}} \subsetneq \mathcal T_X^+\).
\end{proof}

\bibliographystyle{alpha}
\bibliography{wpref}

\bigskip
\footnotesize

Zehao Sha, \textsc{C. N. Yang Institute for Mathematics and Fundamental Physics, Hefei, China}\par\nopagebreak
Email address: \texttt{zhsha@imfp.org.cn}\par\nopagebreak
Homepage: \url{https://ricciflow19.github.io/}

\end{document}